\documentclass[11pt,a4paper]{amsart}

\usepackage[utf8]{inputenc}
\usepackage[T1]{fontenc}
\usepackage{lmodern}
\usepackage[dvipsnames]{xcolor}

\usepackage{geometry}
\newgeometry{inner=2.5cm, outer=2.5cm, top=3cm, bottom=4cm, headsep=1cm}

\usepackage{amsmath,amssymb,amsthm,amsfonts,eucal}
\usepackage[all]{xy}
\usepackage{mathtools}
\usepackage{latexsym}
\usepackage{tikz}

\usetikzlibrary{cd, matrix, arrows, decorations.markings}

\usepackage{enumitem}
\setenumerate{listparindent=\parindent}
\setlist[enumerate]{
	font=\normalfont,
	label=(\roman*),
	topsep=0pt,
	itemsep=-0.3ex,
	partopsep=1ex,
	parsep=1ex}

\usepackage[pdftex, draft=false]{hyperref} 
\hypersetup{
	pdfauthor = {Kevin Aguyar Brix, Alexander Mundey, and Adam Rennie},
	pdftitle = {Splittings for C*-correspondences and strong shift equivalence},
	linktocpage = true,
	colorlinks = true,
	linkcolor = blue,
	citecolor = blue 
}
\usepackage[noabbrev, capitalise, nameinlink]{cleveref}

\newcommand{\bma}{\left(\begin{array}{cc}}
\newcommand{\ema}{\end{array}\right)}
\newcommand{\bca}{\left(\begin{array}{c}}
\newcommand{\eca}{\end{array}\right)}

\newcommand{\cC}{\ensuremath{\mathcal{C}}}
\newcommand{\cD}{\ensuremath{\mathcal{D}}}

\newcommand{\cO}{\ensuremath{\mathcal{O}}}

\newcommand{\cR}{\ensuremath{\mathcal{R}}}

\newcommand{\cT}{\ensuremath{\mathcal{T}}}

\newcommand{\bZ}{\ensuremath{\mathbb{Z}}}

\newcommand{\bT}{\ensuremath{\mathbb{T}}}

\newcommand{\bO}{\ensuremath{\mathbb{O}}}

\numberwithin{equation}{section} 

\theoremstyle{plain} 
\newtheorem{thm}{Theorem}[section]
\newtheorem{lemma}[thm]{Lemma}
\newtheorem{prop}[thm]{Proposition}

\theoremstyle{definition} 
\newtheorem{defn}[thm]{Definition}
\newtheorem{example}[thm]{Example}
\newtheorem{examples}[thm]{Examples}
\theoremstyle{remark} 
\newtheorem{rmk}[thm]{Remark}

\DeclareMathOperator{\Ad}{Ad}    
\DeclareMathOperator{\End}{End}   
\DeclareMathOperator{\Fock}{Fock} 
\DeclareMathOperator{\Id}{Id}     
\DeclareMathOperator{\linspan}{span} 
\DeclareMathOperator{\Mult}{Mult} 
\DeclareMathOperator{\supp}{supp} 

\DeclareMathOperator{\Aut}{Aut}
\DeclareMathOperator{\sing}{sing}
\DeclareMathOperator{\src}{src}
\DeclareMathOperator{\fin}{fin}
\DeclareMathOperator{\infi}{inf}
\DeclareMathOperator{\reg}{reg}

\newcommand{\C}{\mathbb{C}}   
\newcommand{\Cc}{\mathcal{C}} 

\newcommand{\K}{\mathcal{K}}  
\newcommand{\N}{\mathbb{N}}   
\newcommand{\ol}{\overline}  
\newcommand{\ox}{\otimes}     

\newcommand{\Pp}{\mathcal{P}} 
\newcommand{\Rr}{\mathcal{R}} 
\newcommand{\T}{\mathcal{T}} 
\newcommand{\Z}{\mathbb{Z}}   
\newcommand{\pairing}[2]{(#1\mathbin{|}#2)} 
\newcommand{\stroke}{\mathbin|}     

\def\pairL_#1(#2|#3){{}_{#1}(#2\stroke#3)} 
\def\pairR(#1|#2)_#3{(#1\stroke#2)_{#3}} 
\def\scal<#1|#2>{\langle#1\stroke#2\rangle} 

\makeatletter
\newcommand{\tpmod}[1]{{\@displayfalse\pmod{#1}}}
\makeatother

\newbox\ncintdbox \newbox\ncinttbox 
	\setbox0=\hbox{$-$}
	\setbox2=\hbox{$\displaystyle\int$}
	\setbox\ncintdbox=\hbox{\rlap{\hbox
		to \wd2{\hskip-.125em \box2\relax\hfil}}\box0\kern.1em}
	\setbox0=\hbox{$\vcenter{\hrule width 4pt}$}
	\setbox2=\hbox{$\textstyle\int$}
	\setbox\ncinttbox=\hbox{\rlap{\hbox
		to \wd2{\hskip-.175em \box2\relax\hfil}}\box0\kern.1em}

\usepackage{comment}

\title{\vspace*{-1pc}%
	Splittings for C*-correspondences and strong shift equivalence}

\author[K. A. Brix]{Kevin Aguyar Brix}
\email{kabrix.math@fastmail.com}
\address[K. A. Brix]{School of Mathematics and Statistics, University of Glasgow, United Kingdom}
\curraddr{Department of Mathematics and Computer Science, University of Southern Denmark, 5230 Odense, Denmark}
\author[A. Mundey]{Alexander Mundey}
\email{amundey@uow.edu.au}
\author[A. Rennie]{Adam Rennie}
\email{renniea@uow.edu.au}
\address[A. Mundey and A. Rennie]{School of Mathematics and Applied Statistics, University of Wollongong, Australia}

\thanks{K.A.B. was supported by a Carlsberg Foundation Internationalisation Fellowship and a DFF-International Postdoc (case number 1025--00004B).
A.M. was supported by ARC Project DP200100155 and University of Wollongong RevITAlising Research Grant IV036.
A.M. and A.R. thank Bram Mesland and Aidan Sims for useful discussions. 
We also thank Jason DeVito for helpful advice via the Mathematics StackExchange, and Paige Riddiford for careful proofreading.
We thank Adam Dor-On for bringing to our attention a gap in our proof of \cref{thm:sse-equivariant-morita-equivalence}, see \cref{rmk:BDR} for a fix of the problem.}

\makeatletter
\@namedef{subjclassname@2020}{%
  \textup{2020} Mathematics Subject Classification}
\makeatother


\begin{document}

\maketitle

\vspace{-2pc}

\begin{abstract}
 We present an extension of the notion of in-splits from symbolic dynamics to topological graphs and, more generally, to $C^*$-correspondences. We demonstrate that in-splits provide examples of strong shift equivalences of $C^*$-correspondences.
Furthermore, we provide a streamlined treatment of Muhly, Pask, and Tomforde's proof that any strong shift equivalence of regular $C^*$-correspondences induces a (gauge-equivariant) Morita equivalence between Cuntz--Pimsner algebras.
For topological graphs, we prove that in-splits induce diagonal-preserving gauge-equivariant $*$-isomorphisms in analogy with the results for Cuntz--Krieger algebras. Additionally, we examine the notion of out-splits for $C^*$-correspondences.
\end{abstract}

\parskip=6pt
\parindent=0pt

\section{Introduction}
\label{sec:intro}

This paper studies noncommutative dynamical systems---defined as $C^*$-correspondences over not necessarily commutative $C^*$-algebras---
building on previous work \cite{Pimsner,Katsura2004-trans,Katsura2004-jfa,MPT,Kakariadis-Katsoulis,Doron-Eilers-Geffen,Carlsen-Doron-Eilers}.
Inspired by classical constructions of state splittings in symbolic dynamics \cite{Williams1973}, 
we introduce in-splits and out-splits for $C^*$-correspondences. We prove that these operations change the $C^*$-correspondence,
but leave the abstract dynamical system invariant, up to a notion of strong shift equivalence (conjugacy) as defined by Muhly, Pask, and Tomforde.
This strong shift equivalence is reflected in the associated Cuntz--Pimsner $C^*$-algebras as gauge-equivariant Morita equivalence.

Symbolic dynamics \cite{Lind-Marcus}  is a powerful tool in the study of smooth dynamical systems 
(such as toral automorphisms or Smale's Axiom A diffeomorphisms) that works by discretising time using shift spaces.
Every subshift of finite type can be represented by a finite directed graph. The \emph{conjugacy problem} for subshifts of finite type is fundamental: \emph{when are two shifts of finite type the same?}
Williams \cite{Williams1973}  showed that two subshifts of finite type are conjugate
if and only if the adjacency matrices $\mathsf{A}$ and $\mathsf{B}$ of their graph representations are \emph{strong shift equivalent}.
That is, there are adjacency matrices $\mathsf{A} = \mathsf{A_1,\dots,A_n}= \mathsf{B}$ such that for each $i=1,\dots,n-1$ 
there are rectangular matrices with nonnegative integer entries $\mathsf{R}$ and $\mathsf{S}$
such that $\mathsf{A_i} = \mathsf{RS}$ and $\mathsf{SR} = \mathsf{A_{i+1}}$.

Williams' motivation was the observation that state splittings of graph representations change the graph but leave the associated shift space  invariant up to conjugacy.
The data of a  state splitting is reflected in matrices $\mathsf{R}$ and $\mathsf{S}$ as above,
and Williams proved the \emph{decomposition theorem}: 
any conjugacy is a finite composition of elementary conjugacies coming from state splittings.
Deciding whether two subshifts are conjugate can be difficult in practice, 
and it is an open problem in symbolic dynamics  to determine whether strong shift equivalence is decidable.

In \cite{Cuntz-Krieger1980}, Cuntz and Krieger associated a $C^*$-algebra $\cO_{\mathsf{A}}$, now known as a Cuntz--Krieger algebra, to a subshift with adjacency matrix $A$
and showed that it is a universal simple \mbox{$C^*$-algebra} when $\mathsf{A}$ is irreducible and not a permutation.
The $C^*$-algebra $\cO_{\mathsf{A}}$ comes equipped with an action of the circle group $\bT$---the \emph{gauge action}---and a canonical commutative subalgebra---the \emph{diagonal}.
Cuntz and Krieger proved that conjugate subshifts induce Morita equivalent Cuntz--Krieger algebras.

Recently, Carlsen and Rout \cite{Carlsen-Rout} completed the picture: 
$\mathsf{A}$ and $\mathsf{B}$ are strong shift equivalent if and only if there is a $*$-isomorphism $\Phi \colon \cO_{\mathsf{A}} \ox \K \to \cO_{\mathsf{B}}\ox \K$
that is both gauge-equivariant and diagonal-preserving ($\K$ is the $C^*$-algebra of compact operators on separable Hilbert space).
Cuntz--Krieger algebras have been generalised in many ways (e.g. directed graphs and their higher-rank analogues, see \cite{Raeburn} and references therein),
and we emphasise Pimsner's construction from a $C^*$-correspondence \cite{Pimsner},
later refined by Katsura \cite{Katsura2004-jfa},
and applied by Katsura to his topological graphs \cite{Katsura2004-trans}.

We mention in passing that there are other moves on graphs:
Parry and Sullivan's \emph{symbol expansions} \cite{Parry-Sullivan} and the \emph{Cuntz splice} both related to flow equivalence 
as well as more advanced moves \cite{Eilers-Ruiz} which were utilised in the geometric classification of all unital graph $C^*$-algebras \cite{ERRS2016}. We leave open whether these moves have analogues for correspondences.

In the general setting of $C^*$-correspondences (a right Hilbert $C^*$-module with a left action \cite{Lance}), we do not have access to a notion of conjugacy,
but Muhly, Pask, and Tomforde \cite{MPT} introduced strong shift equivalence in direct analogy with Williams' work.
For regular $C^*$-correspondences, they showed that the induced Cuntz--Pimsner algebras are Morita equivalent
(we verify that this Morita equivalence is in fact gauge-equivariant in the sense of \cite{Combes}).
It is an interesting open problem whether the weaker notion of \emph{shift equivalence} introduced in \cite{Kakariadis-Katsoulis} (see also \cite{Carlsen-Doron-Eilers})
also implies gauge-equivariant Morita equivalence.

For directed graphs, an in-split is a factorisation of the range map, and the range map induces the left action on the graph correspondence. 
An in-split of a general correspondence is then formulated as a factorisation of the left action subject to natural conditions.
Similarly, an out-split is a factorisation of the source map which is reflected in the right-module structure of the graph correspondence,
and we define an out-split of a general correspondence accordingly, although this appears less natural than the in-split.
Our notions of splittings of correspondences provide examples of strong shift equivalences. They exhibit the same asymmetry as in the classical setting (cf.~\cite{Bates-Pask}):
an out-split induces a gauge-equivariant Morita equivalence, while an in-split induces a gauge-equivariant $*$-isomorphism of Cuntz--Pimsner algebras.
We leave open the problem of whether an arbitrary strong shift equivalence of correspondences is a composition of splittings.

We specialise our splittings to the case of topological graphs and in this case the analogy with directed graphs is almost complete. 
For general `non-commutative dynamics' defined by \mbox{$C^*$-correspondences} over not-necessarily commutative $C^*$-algebras, 
the analogy is as complete as it can be. 
It is unreasonable to expect complete characterisations of strong shift equivalence in terms of Cuntz--Pimsner algebras 
akin to the Carlsen--Rout result, due to the lack of a diagonal subalgebra for a general correspondence.

In \cref{sec:oncemoreforthecheapseats} we recall what we need about $C^*$-modules, correspondences, and their associated $C^*$-algebras. 
Along the way we provide some proofs for results that seem to be missing from the literature. 
\cref{sec:sse} recalls strong shift equivalence of correspondences and refines the main result of Muhly, Pask, and Tomforde \cite[Theorem 3.14]{MPT}.
In-splits for topological graphs and general correspondences are introduced in \cref{sec:insplit}. 
Within this section we also extend the idea of diagonal subalgebra to topological graphs and show that the gauge equivariant $*$-isomorphisms between a topological graph 
correspondence and any of its in-splits is diagonal-preserving.
Finally, \cref{sec:out-damn-split} defines and gives the basic properties of non-commutative out-splits.

\section{Correspondences and Cuntz--Pimsner algebras}
\label{sec:oncemoreforthecheapseats}

In this preliminary section we provide background information and establish notation for what we need to know about $C^*$-correspondences and their $C^*$-algebras 
(Toeplitz--Pimsner algebras and Cuntz--Pimsner algebras), frames, and topological graphs.

\subsection{\texorpdfstring{$C^*$}{C*}-modules and correspondences}
\label{subsec:mods-coz}

We follow conventions of~\cite{Lance}
for $C^*$-modules, and Pimsner \cite{Pimsner} and Katsura \cite{Katsura2004-jfa} for the algebras defined by $C^*$-correspondences. 

A right Hilbert $A$-module $X_A$ is a right module over a $C^*$-algebra $A$ equipped with an $A$-valued inner product $(\cdot \mid \cdot)_A$
such that $X_A$ is complete with respect to the norm induced by the inner product. The module $X_A$ is \emph{full} if $\ol{(X_A \mid X_A)}_A = A$.
We denote the $C^*$-algebra of adjointable operators on $X_A$ by $\End_A(X)$, 
the $C^*$-ideal of generalised compact operators by $\End^0_A(X)$, and the finite-rank operators by $\End_A^{00}(X)$.
The finite-rank operators are generated by rank-one operators $\Theta_{x,y}$ satisfying $\Theta_{x,y}(z) = x \cdot(y \mid z)_A$,
for all $x,y,z \in X_A$.

\begin{defn}
\label{defn:correspondence}
Let $X_B$ be a right Hilbert $B$-module, and let $\phi_X\colon A\to \End_B(X)$ be a \mbox{$*$-homomorphism}. 
The data $(\phi_X,{}_A X_B)$ is called an $A$--$B$-\emph{correspondence} (or just a correspondence),
and if $\phi_X$ is understood we will write ${}_AX_B$. 
If $A = B$ we refer to $(\phi_X, {}_A X_A)$ as a correspondence \emph{over} $A$.

A correspondence $(\phi_X, {}_AX_B)$ is \emph{nondegenerate} if $\overline{\phi_X(A)X}=X$,
and following~\cite[Definition 3.1]{MPT}, we say the correspondence is \emph{regular} if the left action is \emph{injective} (i.e. $\ker(\phi_X)=\{0\}$) 
and \emph{by compacts} (i.e. $\phi_X(A)\subseteq \End^0_B(X)$).
\end{defn}

Throughout we assume that $A$ and $B$ are both $\sigma$-unital $C^*$-algebras and that all Hilbert modules are countably generated, although many of our results do not critically rely on these assumptions.

There is a natural notion of morphism between correspondences.

\begin{defn}\label{def:cozzymorph}
Let $(\phi_X, {}_AX_A)$ and $(\phi_Y, {}_BY_B)$ be correspondences.
A \emph{correspondence morphism} $(\alpha,\beta) \colon (\phi_X, {}_{A} X_A) \to (\phi_Y, {}_{B} Y_B)$ consists of 
a $*$-homomorphism $\alpha \colon A \to B$ and a linear map $\beta \colon X \to Y$ satisfying:
\begin{enumerate}
  \item $(\beta(\xi) \mid \beta(\eta))_B = \alpha((\xi \mid \eta)_A) $ for all $\xi,\, \eta \in X$; \label{i:cozzymorph1}
  \item $ \beta(\xi \cdot a) = \beta(\xi) \cdot \alpha(a)$, for all $a \in A$ and $\xi \in X$; and \label{i:cozzymorph2}
  \item $ \beta(\phi_X(a) \xi) = \phi_Y(\alpha(a)) \beta(\xi)$, for all $a \in A$ and $\xi \in X$. \label{i:cozzymorph3}
\end{enumerate}
A correspondence morphism is \emph{injective} if $\alpha$ is injective (in which case $\beta$ is isometric) and it is a \emph{correspondence isomorphism} if $\alpha$ and $\beta$ are isomorphisms. 
Composition of morphisms is defined by $(\alpha,\beta) \circ (\alpha',\beta') = (\alpha \circ \alpha', \beta \circ \beta')$.
If $(\phi_Y, {}_BX_B) = ({\Id}_B, {}_B B_B)$ is the identity correspondence \cite{EKQR} over the $C^*$-algebra $B$, then we call $(\alpha,\beta)$ a \emph{representation} of  $(\phi_X,{}_A X_A)$ in $B$.
\end{defn}

A representation $(\alpha,\beta)$ of a $C^*$-correspondence $(\phi_X, X_A)$ is said to \emph{admit a gauge action} 
if there is a strongly continuous action $\gamma^{(\alpha,\beta)}$ of $\bT$ on $C^*(\alpha,\beta) \coloneqq C^*(\alpha(A) \cup \beta(X_A))$---
the $C^*$-algebra generated by the image of $(\alpha,\beta)$ in $B$---by $*$-automorphisms such that
$\gamma^{(\alpha,\beta)}_z(\alpha(a)) = \alpha(a)$ for all $a\in A$, and $\gamma^{(\alpha,\beta)}_z(\beta(x)) = z \beta(x)$ for all $x\in X$.

\begin{defn}
The \emph{Toeplitz algebra} $\cT_X$ of a $C^*$-correspondence $(\phi,{}_A X_A)$ is the universal $C^*$-algebra for representations of $(\phi_X, {}_A X_A)$ in the following sense. 
There exists a representation $(\underline{\iota}_A, \underline{\iota}_X) \colon (\phi_X,{}_A X_A) \to \cT_X$ such that  $\cT_X = C^*(\underline{\iota}_A,\underline{\iota}_X)$, 
and for any other representation $(\alpha,\beta) \colon (\phi_X,{}_A X_A) \to B$ in a $C^*$-algebra $B$, 
there is a unique $*$-homomorphism $\alpha \times \beta \colon \cT_X \to B$ 
such that $(\alpha \times \beta) \circ \underline{\iota}_A = \alpha$ and $(\alpha \times \beta) \circ \underline{\iota}_X= \beta$.
\end{defn}

To a correspondence $({\phi_X}, {}_A X_A)$ we associate its \emph{covariance ideal}
\[
	J_{\phi_X} \coloneqq \phi_X^{-1}(\End_A^0(X)) \cap \ker(\phi_X)^{\perp} ,
\]
which is an ideal in $A$ (cf.~\cite[Definition 3.2]{Katsura2004-jfa}). 
The covariance ideal is the largest ideal of $A$ such that the restriction of $\phi_X$ to it is both injective and has image contained in $\End_A^0(X)$. 
We will often consider covariant morphisms (defined below) which respect the covariance ideal.

A correspondence morphism $(\alpha,\beta) \colon (\phi_X, {}_A X_A) \to (\phi_Y, {}_B Y_B)$
induces a $*$-homomorphism of compacts $\beta^{(1)}\colon \End^0_A(X) \to \End_B^0(Y)$ satisfying $\beta^{(1)}(\Theta_{x_1,x_2}) = \Theta_{\beta(x_1),\beta(x_2)}$ for all $x_1,x_2\in X$. 

\begin{defn}
A morphism $(\alpha,\beta) \colon (\phi_X, {}_A X_A) \to (\phi_Y, {}_B Y_B)$ is \emph{covariant} if 
\[
\beta^{(1)} \circ \phi_X(c) = \phi_Y \circ \alpha(c) \quad \text{ for all } c \in J_{\phi_X}.
\]
In particular, we must have $\alpha(J_{\phi_X}) \subseteq J_{\phi_Y}$. 
If $(\phi_Y, {}_BX_B) = ({\Id}_B, {}_B B_B)$ is the identity correspondence over $B$, then we call $(\alpha,\beta)$ a \emph{covariant representation} of  $(\phi_X,{}_A X_A)$ in $B$.
\end{defn}

\begin{defn}
The \emph{Cuntz--Pimsner algebra} $\cO_X$ of a $C^*$-correspondence $(\phi,{}_A X_A)$ is the universal $C^*$-algebra for covariant representations of $(\phi_X, {}_A X_A)$ in the following sense. 
There exists a universal covariant representation $(\iota_A,\iota_X) \colon (\phi_X,{}_A X_A) \to \cO_X$ such that $\cO_X = C^*(\iota_A,\iota_X)$, and 
for any other covariant representation $(\alpha,\beta) \colon (\phi_X,{}_A X_A) \to B$ on a $C^*$-algebra $B$,
there is a unique $*$-homomorphism $\alpha \times \beta \colon \cO_X \to B$ such that $(\alpha \times \beta) \circ 
\iota_A = \alpha$ and $(\alpha \times \beta) \circ \iota_X = \beta$.
\end{defn}

The universal covariant representation $(\iota_A, \iota_X)$ admits a gauge action $\gamma^X\colon \mathbb{T}\curvearrowright \cO_X$ that we shall refer to as the \emph{canonical gauge action}.

\begin{lemma}\label{lem:induced-gauge-invariant-homomorphism} 
  Let $(\alpha,\beta)\colon (\phi_X, {}_AX_A) \to (\phi_Y, {}_BY_B)$ be a covariant correspondence morphism,
  and let $(\iota_A, \iota_X)$ and $(\iota_B,\iota_Y)$ be universal covariant representations of $\cO_X$ and $\cO_Y$, respectively.
  Then there is an induced gauge-equivariant $*$-homomorphism $\alpha\times \beta\colon \cO_X \to \cO_Y$
  satisfying 
  \[(\alpha \times \beta)\circ \iota_A = \iota_B\circ \alpha \quad \text{and} \quad (\alpha \times \beta)\circ \iota_X = \iota_Y\circ \beta.\]
  If $\alpha$ is injective, then $\alpha\times \beta$ is injective.
\end{lemma}
\begin{rmk}
The relation $(\alpha \times \beta) \circ \iota_X^{(1)} = \iota_{Y}^{(1)} \circ \beta$ also follows easily from the lemma and the definition of the induced ${}^{(1)}$ maps on compacts. 
\end{rmk}

\begin{proof}
  The composition $(\iota_B, \iota_Y)\circ (\alpha, \beta)$ is a covariant representation of $(\phi_X, X_A)$ on $\cO_Y$,
  so by the universal property (and a slight abuse of notation) there is a $*$-homomorphism $\alpha\times \beta\colon \cO_X \to \cO_Y$ satisfying
  $(\alpha \times \beta)\circ \iota_A = \iota_B\circ \alpha$ and $(\alpha \times \beta)\circ \iota_X = \iota_Y\circ \beta$.
  If $a\in A$, then 
  \[
    (\alpha\times \beta)\circ \gamma^X_z(\iota_A(a)) = \iota_B\circ \alpha(a) = \gamma^Y_z\circ (\alpha\times \beta)(\iota_A(a)),
  \]
  for all $z\in \mathbb{T}$, and if $x\in X_A$, then
  \[
    (\alpha\times \beta)\circ \gamma^X_z(\iota_X(x)) = z (\alpha\times \beta)(\iota_X(x)) = \gamma^Y_z\circ (\alpha\times \beta)(\iota_X(x)),
  \]
  for all $z\in \mathbb{T}$.
  This shows that $\alpha\times \beta$ is gauge-equivariant.
  If $\alpha$ is injective, then $(\iota_A\circ \alpha, \iota_X\circ \beta)$ is an injective representation that admits a gauge action
  so $\alpha\times \beta$ is injective by the gauge invariant uniqueness theorem~\cite[Theorem 6.4]{Katsura2004-jfa}.
\end{proof}

To talk about Morita equivalence we isolate a special kind of correspondence.
\begin{defn}
	\label{defn:Morita-equiv}
	An $A$--$B$-imprimitivity bimodule between $C^*$-algebras $A$ and $B$
	is a correspondence $(\phi,{}_AX_B)$ with an additional left $A$-valued inner product such that the right $B$ action is adjointable for the left inner product, and  $X$ is full as a left and as a
	right module. Moreover
	$$
	\phi({}_A(x|y))z=x\cdot (y|z)_B\quad x,\,y,\,z\in X.
	$$
	If such an imprimitivity bimodule exists then $A$ and $B$ are \emph{Morita equivalent}.
	\end{defn}
	
	There is also a group-equivariant version of Morita equivalence due to Combes, \cite{Combes}. 
    To describe equivariant Morita equivalence and the gauge action of the circle on Cuntz--Pimsner algebras we recall some definitions and results.

\begin{defn} Let $G$ be a locally compact Hausdorff group and let $A$ be a $G$-$C^*$-algebra with strongly continuous action $\alpha \colon G \to \Aut(A)$.
An \emph{action of $G$ on an $A$-module $X_A$} is a  strongly continuous action $g \mapsto U_g$ of $G$ on $X_A$ by $\C$-linear isometries such that
	\begin{enumerate}
		\item $U_g (x\cdot a) = U_g(x)\alpha_g(a)$ for all $x \in X$ and $a \in A$; and\label{1}
		\item $(U_gx \mid U_gy)_A = \alpha_g ((x \mid y)_A)$ for all $x,\,y \in X$.
	\end{enumerate}
	If $(\phi,{}_B X_A)$ is a correspondence and $B$ is a $G$-$C^*$-algebra with action $\beta:G\to \Aut(B)$, then $U$ is an \emph{action on the correspondence} if $U$ is an action on $X_A$ and in addition $U_g \phi(b) = \phi(\beta_g(b))U_g$  for all $b \in B$. The action is \emph{covariant}, if in addition $\beta_g(J_X) = J_X$. 
\end{defn}

\begin{rmk}
	The operators $U_g$ on $X_A$ are typically not $A$-linear due to condition \ref{1}.
\end{rmk}

\begin{lemma}
	If $(U,\alpha)$ is an action of $G$ on the right module $X_A$, then there is an induced strongly continuous action $\ol{\alpha} \colon G \to \Aut(\End_A^0(X))$ defined by $\ol{\alpha}_g(T) \coloneqq \Ad_{U_g} (T) = U_g T U_{g^{-1}}$. For rank-1 operators $\ol{\alpha}_g(\Theta_{x,y}) = \Theta_{U_gx,U_gy}$. 
\end{lemma}

An action of a group on a correspondence induces a ``second quantised'' action on both the associated Toeplitz and Cuntz--Pimsner algebras. 
This is an immediate consequence of the universal properties of both Toeplitz and Cuntz--Pimsner algebras. 
\begin{lemma}[cf. \cite{LN}]
\label{lem:LN}
	If $(U,\alpha)$ if an action of $G$ on an $A$-correspondence  $(\phi,{}_A X_A)$, then there is an induced action $\sigma \colon G \to \Aut(\cT_X)$ on the Toeplitz-Pimsner algebra such that
	\[
	\sigma_g(\underline{\iota}_A(a)) = \underline{\iota}_A(\alpha_g(a)) \quad \text{and} \quad \sigma_g(\underline{\iota}_X(x)) = \underline{\iota}_X({U_gx})
	\]
	for all $g \in G$, $a \in A$, and $x \in X$. If the action $(U,\alpha)$ is covariant, then $\sigma$ descends to an action $\sigma \colon G \to \Aut(\cO_X)$. 
\end{lemma}

\begin{example}
The action of the circle $\mathbb{T}$ on a correspondence $(\phi,{}_AX_A)$ defined by
\[
U_z(x)=zx,\qquad \alpha_z(a)=a,\qquad x\in X,\ \ a\in A,\ \ z\in \mathbb{T}
\]
happens to have each $U_z$ adjointable, and induces the gauge actions on $\T_X$ and $\cO_X$.
\end{example}
	
\begin{defn}
\label{def:gauge-ME}
Let $A$ and $B$ be $C^*$-algebras and suppose that $\gamma^A\colon G\curvearrowright A$ and $\gamma^B\colon G\curvearrowright B$ are strongly continuous actions of a locally compact Hausdorff group $G$.
Following Combes~\cite{Combes}, we say that $\gamma^A$ and $\gamma^B$ are \emph{Morita equivalent} if: 
\begin{enumerate}
\item there is an $A$-$B$-imprimitivity bimodule ${}_AX_B$; 
\item there is a strongly continuous action of $G$ on  $X$ by $\C$-linear isometries $U_g$; and
\item the above action of $G$ on $X$ restricts to an action on $A = \End^0_B(X)$ which coincides with $\gamma_A$,
  and it restricts to an action on $B = \End^0_A(X)$ which coincides with $\gamma_B$. 
\end{enumerate}
Equivalently, $\gamma^A$ and $\gamma^B$ are Morita equivalent if there exists a $C^*$-algebra $C$ such that $A$ and $B$ are (isomorphic to)
complementary full corners in $C$, and $C$ admits an action $\gamma^C$ such that $\gamma^A$ is $\gamma^C|_A$ and $\gamma^B$ is $\gamma^C|_B$, cf, \cite[Section 4]{Combes}.
\end{defn}

\subsection{Frames}
\label{subsec:frame}

An important technical and computational tool for Hilbert {$C^*$-modules} is the concept of a frame. 
This is as close as one can get to an orthonormal basis in a $C^*$-module, and it serves similar purposes.
In fact, Kajiwara, Pinzari, and Watatani refer to frames as bases, see~\cite{KPW04}.
In the signal analysis literature, see for instance \cite{FL,L}, what we call a frame is also known as a \emph{standard normalised tight frame}.

\begin{defn}
  \label{defn:frame}
  Let $X_A$ be a right $A$-module.
  A (right) \emph{countable frame} for $X_A$ is a sequence $(x_j)_{j\in \N}$ in $X_A$ such that $\sum_{j =1}^{\infty}\Theta_{x_j,x_j}$ converges strictly to the identity operator in $\End_A(X)$. 
  Equivalently, we have $x = \sum_{j=1}^\infty \Theta_{x_j, x_j} x$ for all $x \in X_A$ with the sum converging in norm.
\end{defn}
For the strict topology, see \cite{Lance}, but for our purposes it is enough to know that the strict topology coincides with the $*$-strong topology on bounded sets.

If $(x_j)_{j \in \N}$ is a frame for $X_A$, then $X_A$ is generated as a right $A$-module by $x_j$, so $X$ is countably generated.
Conversely, any countably generated $C^*$-module over a $\sigma$-unital $C^*$-algebra $A$ admits a countable frame, cf.~\cite[Proposition 2.1]{KPW04}. 

The following result is well-known to experts, but we were unable to find a reference. 
As the proof is non-trivial, we include it for completeness. 
We  thank Bram Mesland for helpful suggestions.

\begin{prop}\label{prop:tensor_product_of_frames}
	Let $X_A$ be a countably generated right Hilbert $A$-module and let $(\phi,{}_A Y_B)$ be a countably generated $A$--$B$-correspondence. 
    Let $(x_i)_{i \in \N}$ be a countable frame for $X_A$ and let $(y_j)_{j \in \N}$ be a countable frame for $Y_B$. Then $(x_i \otimes y_j)_{i,j \in \N}$ is a countable frame for $X \otimes_{\phi} Y$.
\end{prop}

To prove \cref{prop:tensor_product_of_frames} we require some technical lemmas.
\begin{lemma}\label{lem:postiveorder}
  Let $A$ be a $C^*$-algebra. Suppose $a,b \in A$ are positive elements such that $a \le b \le 1$ in the minimal unitisation $A^+$. Then for each $h \in A$ we have $\|ah - h\| \le \|bh - h\|$.
\end{lemma}
\begin{proof}
	This follows from the calculation
	\begin{align*}
		\|ah-h\|^2 &= \|(a-1)h\|^2
		= \|h^* (a-1)^2 h\|
        = \sup\{ \phi(h^* (a-1)^2 h) \}\\
		&\le \sup\{ \phi(h^*(b-1)^2 h) \}
		=\|bh-h\|^2,
	\end{align*}
	where the supremum is taken over all states $\phi$ on $A$ and $1 \in A^+$.
\end{proof}

\begin{lemma}\label{lem:compactbound}
Let $X_A$, $(\phi,{}_A Y_B)$, $(x_i)_{i \in \N}$, and $(y_i)_{i \in \N}$ be as in the statement of \cref{prop:tensor_product_of_frames}. Then for each $N,M \in \N$,
	\[
	\Big\| \sum_{i=1}^M \sum_{j=1}^N \Theta_{x_i \otimes y_j,x_i \otimes y_j}\Big\| \le 1.
	\]
\end{lemma}

\begin{proof}
	Let $\ell^2(Y) \coloneqq \ell^2(\N) \ox_{\C} Y$ with the standard right $B$-module structure.	
    For each $N \in \N$, we wish to define a linear, right $B$-linear, map $\psi_N : X \ox_{\phi} Y \to \ell^2(Y)$ on elementary tensors by
	\[
	(\psi_N (x \otimes y))_i = \begin{cases}
		\phi((x_i \mid x)_A)  y & \text{if } i \le N;\\
		0 & \text{otherwise.}
	\end{cases}
	\]
    Observe that 
	\begin{align*}
		\|\psi_N (x \otimes y)\|_{\ell^2(Y)}^2 &= \Big\| \sum_{i=1}^N (\phi((x_i \mid x)_A) \cdot y \mid (\phi((x_i \mid x)_A) \cdot y)_B \Big\|\\
		&=\Big\| \sum_{i=1}^N (y \mid \phi( (x \mid x_i)_A (x_i \mid x)_A ) y)_B \Big\|
		= \Big\|  \Big(y \Bigm\vert \phi \Big(x \Bigm\vert \sum_{i=1}^N \Theta_{x_i,x_i} x\Big)_A y \Big)_B \Big\|\\
		&\le \|x\|^2 \|y\|^2 \Big\|\sum_{i=1}^N \Theta_{x_i,x_i}\Big\|
		\le  \|x\|^2 \|y\|^2,
	\end{align*}
	so that $\|\psi_N\| \le 1$. 
    Hence, $\psi_N$ extends to a bounded linear map $\psi_N \colon X \otimes_\phi Y \to \ell^2(Y)$.
    Observe that $\psi_N$ is adjointable with adjoint $\psi_N^* ((z_i)_i) = \sum_{i=1}^N x_i \otimes z_i.$ 
    Embedding $\psi_N$ in the ``bottom left'' corner of the $C^*$-algebra $\End_B((X \ox_{\phi} Y) \oplus \ell^2(Y))$ shows that $\|\psi_N^*\| = \|\psi_N\| \le 1$.
	
	Now for each $M \in \N$ let $T_M = \sum_{j=1}^M \Theta_{y_j,y_j}$. 
    Observe that $T_M$ acts diagonally on $\ell^2(Y)$ and that as an operator on $Y_B$ we have $\|T_M\|_{\End_B(Y)} \le 1$. 
    Then for $(z_i)_i \in \ell^2(Y)$,
	\begin{align*}
		\|T_M ((z_i)_i)\|^2
		&=\Big\| \sum_{n=1}^\infty (T_M z_i \mid T_M z_i)_B\Big\| \le   \Big\|\sum_{n=1}^\infty \|T_M\|_{\End_B(Y)}^2 ( z_i \mid z_i)_B\Big\| \\
		&\le \|T_M\| _{\End_B(Y)}^2\|(z_i)_i\|^2 \le \|(z_i)_i\|^2,
	\end{align*}
	where the first inequality follows from \cite[Proposition 1.2.]{Lance}.  Thus, $\|T_M\|_{\End_B(\ell^2(Y))} \le 1$.
	Since we can write 
	\[
	\sum_{i=1}^M \sum_{j=1}^N \Theta_{x_i \otimes y_j,x_i \otimes y_j} = \psi_N^* \circ T_M \circ \psi_N
	\] 
	the result follows.
\end{proof}

\begin{proof}[Proof of \cref{prop:tensor_product_of_frames}]
	It  suffices to show that $(\sum_{(i,j) \in \Sigma} \Theta_{x_i \otimes y_j,x_i \otimes y_j})_{\Sigma \subset\subset \N^2}	$ is an approximate identity for $\End_B^0(X \otimes_\phi Y)$, where the sequence is indexed by finite subsets $\Sigma$ of $\N^2$.	
	Fix $\varepsilon > 0$. We first claim that for each $\xi \in X \ox_{\phi} Y$ there exists $M,N \in \N$ such that 
	\begin{equation}
	\Big\| \sum_{i=1}^M \sum_{j=1}^N \Theta_{x_i \otimes y_j,x_i \otimes y_j} \xi - \xi \Big\| <\varepsilon.
	\label{eq:approx}
	\end{equation}
	It suffices to consider the case where $\xi = \eta \otimes \zeta$ for some $\eta \in X_A$, $\zeta \in Y_B$. Take $M$ large enough so that 
	\[
	\Big\|\sum_{i=1}^M \Theta_{x_i,x_i} \eta - \eta \Big\| < \frac{\varepsilon}{2}
	\]
	and take $N$ large enough so that 
	\[
	\Big\|\sum_{j=1}^N \Theta_{y_j,y_j} \phi((x_i\mid \eta)_A) \zeta - \phi((x_i\mid \eta)_A)  \zeta \Big\| < \frac{\varepsilon}{2M}
	\]
	for all $1 \le i \le M$. It follows that
	\begin{align*}
		&\Big\| \sum_{i=1}^M \sum_{j=1}^N \Theta_{x_i \otimes y_j,x_i \otimes y_j} \xi - \xi \Big\|
		= \Big\|\sum_{i=1}^M \sum_{j=1}^N x_i \otimes y_j \cdot (y_j \mid \phi((x_i \mid \eta)_A) \zeta)_B - \eta \otimes \zeta \Big\|\\
		&\le \Big\|\sum_{i=1}^M \sum_{j=1}^N x_i \otimes y_j \cdot (y_j \mid \phi((x_i \mid \eta)_A) \zeta)_B - \sum_{i=1}^M x_i \otimes \phi((x_i \mid \eta)_A) \zeta \Big\|\\
		&\quad+ \Big\|\sum_{i=1}^M x_i \otimes \phi((x_i \mid \eta)_A ) \zeta - \eta \otimes \zeta\Big\|\\
		&\le \sum_{i=1}^M \|x_i\| \,\Big\|\sum_{j=1}^N y_j \cdot (y_j \mid \phi((x_i \mid \eta)_A)  \zeta)_B - \phi((x_i \mid \eta)_A)  \zeta \Big\| + \frac{\varepsilon}{2}
		< \varepsilon.
	\end{align*}	
	We now claim that for each $T \in \End_B^0(X \ox_{\phi} Y)$ there is a sequence $(M_k,N_k)_{k=1}^\infty$ in $\N^2$, with each of $(M_k)_k$ and $(N_k)_k$ strictly increasing, such that 
	\begin{equation}
	\sum_{i=1}^{M_k}\sum_{j=1}^{N_k} \Theta_{x_i \otimes y_j,x_i \otimes y_j} T \to T
	\label{eq:cvge}
	\end{equation} 
	as $k \to \infty$. If $T$ is a rank-one operator, then \eqref{eq:cvge} holds for $T$, as follows from the claim \eqref{eq:approx}, which we have proved. By taking finite sums, the claim is also true for finite-rank $T$.
	
	Fix $\varepsilon > 0$. Suppose that $T \in \End_B^0(X \otimes_\phi Y)$ is arbitrary, and take a finite rank operator $S$ such that $\|T - S\| < \frac{\varepsilon}{3}$. Let $M, N \in \N$ be such that $\| \sum_{i=1}^M \sum_{j=1}^N \Theta_{x_i \otimes y_j,x_i \otimes y_j} S - S\| < \frac{\varepsilon}{3}$. Then \cref{lem:compactbound} implies that,
	\begin{align*}
		&\Big\| \sum_{i=1}^M \sum_{j=1}^N \Theta_{x_i \otimes y_j,x_i \otimes y_j} T - T \Big\| \\
		&\le \Big\| \sum_{i=1}^M \sum_{j=1}^N \Theta_{x_i \otimes y_j,x_i \otimes y_j}\Big\| \|T-S\|
		+ \Big\| \sum_{i=1}^M \sum_{j=1}^N \Theta_{x_i \otimes y_j,x_i \otimes y_j} S - S\Big\| 
		+ \|T - S\|
		<\varepsilon.
	\end{align*}	
    To finish, fix $\varepsilon > 0$, let $T \in \End_C^0(X \otimes_\phi Y)$, and take $K$ large enough so that 
	\[
	\Big\|\sum_{i=1}^{M_K}\sum_{j=1}^{N_K} \Theta_{x_i \otimes y_j,x_i \otimes y_j} T - T\Big\| < \varepsilon.
	\]
	\cref{lem:postiveorder} shows that for any finite set $\Sigma \subseteq \N^2$ with $\{(i,j) \mid 1 \le i \le M_k, 1 \le j \le N_k\} \subseteq \Sigma$ we have $\|\sum_{(i,j) \in \Sigma}\Theta_{x_i \otimes y_j,x_i \otimes y_j} T - T\| < \varepsilon$. Consequently, $(\sum_{(i,j) \in \Sigma} \Theta_{x_i \otimes y_j,x_i \otimes y_j})_{\Sigma \subset\subset \N^2}$ is an approximate identity for $\End_C^0(X \ox_{\phi} Y)$, and $(x_i\ox y_j)_{i,j}$ is a frame for $X\ox_\phi Y$.
\end{proof}

\subsection{Topological graphs}
Topological graphs and their $C^*$-algebras were introduced by Katsura~\cite{Katsura2004-trans} as a  generalisation of directed graphs and their $C$*-algebras. 
Any (partially defined) local homeomorphism on a locally compact Hausdorff (sometimes known as a \emph{Deaconu--Renault system}) space may be interpreted as a topological graph
and, in turn, any topological graph admits a boundary path space whose shift map gives a Deaconu--Renault system. The $C^*$-algebra of the topological graph is  $*$-isomorphic to the $C^*$-algebra of its associated Deaconu--Renault system, cf. \cite{Katsura2021, Katsura2009}.

\begin{defn}
	A \emph{topological graph} $E = (E^0,E^1,r,s)$ is a quintuple consisting of second countable locally compact Hausdorff spaces $E^0$ of \emph{vertices} and $E^1$ of \emph{edges}, 
	together with a continuous \emph{range} map $r \colon E^1 \to E^0$, and a local homeomorphism $s \colon E^1 \to E^0$ called the \emph{source}. 
	
	Two topological graphs $E = (E^0,E^1,r_E,s_E)$, $F = (F^0,F^1,r_F,s_F)$ are (graph) isomorphic if there are homeomorphisms
	\[
	\mu\colon E^0\to F^0\quad\mbox{and}\quad \nu\colon E^1\to F^1
	\]
	such that $\mu\circ s_E=s_F\circ \nu$ and $\mu\circ r_E=r_F \circ \nu$.
	
	If $E^0$ and $E^1$ are countable discrete sets, then $E$ is a \emph{directed graph}. 
\end{defn}

\begin{rmk}
	The term \emph{topological graph} is sometimes also used to refer to the more general notion of a \emph{topological quiver}  (see \cite{MT}).
	In a topological quiver the condition that $s$ is a local homeomorphism is weakened to $s$ being an open map with the additional requirement of a compatible family of measures on the fibres of $s$. 
    We do not work in this generality.
\end{rmk}

In~\cite[Section 4]{Katsura2021}, Katsura studies the boundary path space $E_\infty = (E_\infty^0, E_\infty^1, r_\infty, s_\infty)$ of a topological graph $E$.
This is again a topological graph and $\sigma_E \coloneqq s_\infty \colon E_\infty^1 \to E_\infty^0$ is a partially defined local homeomorphism.
Two topological graphs $E$ and $F$ are then said to be \emph{conjugate} if the Deaconu--Renault systems on their boundary path spaces are conjugate,
i.e. if there is a homeomorphism $h\colon E_\infty^1 \to F_\infty^1$ such that $h\circ \sigma_E = \sigma_F \circ h$ and $h^{-1}\circ \sigma_F = \sigma_E\circ h^{-1}$.

The space of \emph{paths of length $n$} in a topological graph is the $n$-fold fibred product
\begin{align*}
	E^n 
	&\coloneqq E^1 \times_{s,r} \cdots \times_{s,r} E^1
	=\Big\{ e_1e_2\cdots e_n \in \prod^n E^1 \mid s(e_i)=r(e_{i+1}) \Big\}.
\end{align*}
equipped with the subspace topology of the product topology.

\begin{defn}
To a topological graph $E= (E^0,E^1,r_E,s_E)$, Katsura \cite{Katsura2004-trans} associates a $C_0(E^0)$-correspondence $X(E)$ as follows.
Equip $C_c(E^1)$ with the structure of a pre-$C_0(E^0)$--$C_0(E^0)$-correspondence by
\begin{align*}
	x\cdot a(e) &= x(e) a(s(e))\\
	a \cdot x(e) &= a(r(e))x(e)\\
	(x \mid y)_{C_0(E^0)}(v) &= \sum_{s(e) =v} \ol{x(e)}y(e),
\end{align*}
for all $x,y\in C_c(E^1)$, $a\in C_0(E^0)$, $e\in E^1$, and $v\in E^0$.
The completion $X(E)$ is a $C_0(E^0)$--$C_0(E^0)$-correspondence called the \emph{graph correspondence of $E$},
and the Cuntz--Pimsner algebra $\mathcal{O}_{X(E)}$ is 
called the \emph{$C^*$-algebra of the topological graph $E$}.
\end{defn}
We  fix some terminology to discuss regular and singular points of 
topological graphs.
\begin{defn}\label{defn:sets_of_vertices}
	Let $\psi \colon X \to Y$ be a continuous map between locally compact Hausdorff spaces. 
	We consider the following subsets of $Y$:
	\begin{itemize}
		\item \emph{$\psi$-sources}: $Y_{\psi-\src} \coloneqq Y \setminus \overline{\psi(X)}$
		\item \emph{$\psi$-finite receivers}:
		$\begin{aligned}[t]
			Y_{\psi-\fin} \coloneqq \{ y \in Y &: \exists \text{ a precompact open neighbourhood $V$ of $y$}\\ & \quad \text{such that $\psi^{-1}(\overline{V})$ is compact}\}
		\end{aligned}$
		\item \emph{$\psi$-infinite receivers}: $Y_{\psi-\infi} \coloneqq Y \setminus Y_{\psi-\fin}$
		\item \emph{$\psi$-regular set}: $Y_{\psi-\reg} \coloneqq Y_{\psi-\fin} \setminus \overline{Y_{\psi-\src}}$
		\item \emph{$\psi$-singular set}: $Y_{\psi-\sing} \coloneqq Y \setminus Y_{\psi-\reg} = Y_{\psi-\infi} \cup \overline{Y_{\psi-\src}}$. 
	\end{itemize}
\end{defn}

\begin{rmk}
	If $E = (E^0,E^1,r,s)$ is a topological graph then we use the range map $r \colon E^1 \to E^0$ to construct subsets of $E^0$ according to~\cref{defn:sets_of_vertices}. 
	In this context we drop the map $r$ and, for instance, write $E^0_{\reg} = E^0_{r-\reg}$.  
\end{rmk}

A topological graph $E$ is said to be \emph{regular} if $E^0_{\sing} = \varnothing$.

We recall that the behaviour of the left action $\phi \colon C_0(E^0) \to \End_{C_0(E^0)}^0(X(E))$ is reflected in the singular structure of $E$. 
In particular, the covariance ideal is given by $J_{\phi} = C_0(E^0_{\reg})$, so that regular topological graphs induce regular graph correspondences, cf. \cite{Katsura2004-trans}.

A frame for the graph correspondence is relatively easy to describe.

\begin{example}\label{ex:frame_for_topological_graph}
	Let $E = (E^0,E^1,r,s)$ be a topological graph. Since $E^1$ is second countable and locally compact its paracompact; so admits a locally finite cover $\{U_i\}_{i \in I}$ by precompact open sets such that the restrictions $s \colon U_i \to s(U_i)$ are homeomorphisms onto their image. Let $\{\rho_i\}_{i \in I}$ be a partition of unity subordinate to $\{U_i\}_{i \in I}$ and let $x_i = \rho_i^{1/2}$. We claim that $(x_i)_{i \in I}$ is a frame for $X(E)$. For each $x \in C_c(E^1)$ and $e \in E^1$,
	\[
	\sum_{i} (x_i \cdot (x_i \mid x)_{C_0(E^0)}) (e) = \sum_i \sum_{s(f) = s(e)} x_i(e) \ol{x_i(f)} x(f) = \sum_i \rho_i(e) x(e) = x(e).
	\]
	Since $x$ has compact support, finitely many of the $U_i$ cover $\supp(x)$. Hence,
	\begin{align*}
		\Big\|	\sum_{i} (x_i \cdot (x_i \mid x)_{C_0(E^0)}) - x\Big\|^2 = \sup_{v \in E^0} \sum_{s(e)= v}\Big|\sum_{i} (x_i \cdot (x_i \mid x)_{C_0(E^0)}) (e) - x(e)\Big|^2 \to 0.
	\end{align*}
	Since $C_c(E^1)$ is dense in $X(E)$ it follows that $(x_i)_{i \in I}$ is a frame for $X(E)$. 
\end{example}

\section{Strong shift equivalence}
\label{sec:sse}

Strong shift equivalence was introduced by Williams in~\cite{Williams1973} as an equivalence relation on \emph{adjacency matrices}: finite square matrices with nonnegative integral entries
in the context of shifts of finite type \cite{Lind-Marcus}.
Two adjacency matrices $\textsf{A}$ and $\textsf{B}$ are elementary strong shift equivalent if there exist rectangular matrices $\textsf{R}$ and $\textsf{S}$ with nonnegative integral entries
such that 
\[
  \textsf{A} = \textsf{R S} \quad \textrm{and} \quad \textsf{B} = \textsf{S R}. 
\]
This is not a transitive relation. 
To amend this we say that $\textsf{A}$ and $\textsf{B}$ are strong shift equivalent if there are square matrices $\textsf{A} = \textsf{A}_1,\ldots,\textsf{A}_n = \textsf{B}$ such that 
$\textsf{A}_i$ is elementary strong shift equivalent to $\textsf{A}_{i+1}$ for all $i=1,\ldots,n-1$.
The \emph{raison d'\^etre} for this equivalence relation is the following classification theorem due to Williams: recalling that
a shift of finite type may be represented by an adjacency matrix, a pair of two-sided shifts of finite type are topologically conjugate if and only if the adjacency matrices that represent the systems are strong shift equivalent.

Muhly, Pask, and Tomforde~\cite{MPT} introduce \emph{strong shift equivalence} for $C^*$-correspondences, which we recall below. 
They show that the induced Cuntz--Pimsner algebras of strong shift equivalent correspondences are Morita equivalent.
Kakariadis and Katsoulis~\cite{Kakariadis-Katsoulis} later introduced the a priori weaker notion of \emph{shift equivalence} of $C^*$-correspondences,
and similar notions were further studied by Carlsen, Dor-On, and Eilers~\cite{Carlsen-Doron-Eilers}.

\begin{defn}[{\cite[Definition 3.2]{MPT}}]
  Correspondences $(\phi_X, {}_A X_A)$ and $(\phi_Y, {}_B Y_B)$ are \emph{elementary strong shift equivalent} if there are correspondences $(\phi_R, {}_A R_B)$ and $(\phi_S, {}_B S_A)$ such that 
  \[
    X \cong R\ox_B S \quad \textrm{and} \quad Y \cong S\ox_A R.
  \]
  They are \emph{strong shift equivalent} if there are correspondences $X = X_1,\ldots,X_n = Y$ such that $X_i$ is elementary strong shift equivalent to $X_{i+1}$ for all $i=1,\ldots,n-1$.
\end{defn}

In \cite[Remark 3.6]{Carlsen-Doron-Eilers}, Carlsen, Dor-On, and Eilers observe that if adjacency matrices are strong shift equivalent in the sense of Williams,
then their $C^*$-correspondences are also strong shift equivalent in the sense of Muhly, Pask, and Tomforde. 
The converse is still not known.

In this section we show that there is a gauge equivariant Morita equivalence of the Cuntz--Pimsner algebras of strong shift equivalent correspondences in the sense of \cref{def:gauge-ME}. 
In the process, we revisit the Morita equivalence proof of \cite{MPT} and break it into a series of instructive lemmas. 
The first records how Cuntz--Pimsner algebras behave with respect to direct sums of correspondences. 
\begin{lemma}
\label{lem:covariant_sums}
  Let $(\phi_X, {}_A X_A)$ and $(\phi_Y, {}_B Y_B)$ be correspondences.
  The inclusion $(j_A,j_X)$ of $(\phi_X, {}_AX_A)$ into the $A \oplus B$-correspondence $(\phi_{X\oplus Y}, {}_{A\oplus B}X\oplus Y_{A\oplus B})$ is a covariant correspondence morphism
  that induces a gauge-equivariant and injective $*$-homomorphism $j_A \times j_X \colon \cO_X \to \cO_{X\oplus Y}$.
\end{lemma}

\begin{proof}
It is clear that $(j_A,j_X)$ is a correspondence morphism, and for covariance we must show that $j_X^{(1)} \circ \phi_X(c) = \phi_{X \oplus Y} \circ j_A(c)$ for all $c \in J_X$.
Let $(x_i)_i$ be a frame for $X$ and $(y_j)_j$ a frame for $Y$.
A frame for $X \oplus Y$ is given by the direct sum of the frames for $X$ and $Y$.
Let $P_X$ denote the projection in $\End_{A \oplus B}(X \oplus Y)$ onto $X$ so that $P_X = \sum_i \Theta_{j_X(x_i),j_X(x_i)}$, with the sum taken in the strict topology. 
It follows that 
\begin{align*}
    j_X^{(1)} \circ \phi_X(c) 
    &= \sum_i \Theta_{j_X(\phi_X(c)x_i),j_X(x_i)} 
    = \phi_{X \oplus Y}(j_A(c)) \sum_i \Theta_{j_X(x_i),j_X(x_i)} \\
    &= \phi_{X \oplus Y}(j_A(c)) P_X 
    = \phi_{X \oplus Y} \circ j_A(c),
\end{align*}
for all $c \in J_X$.
\cref{lem:induced-gauge-invariant-homomorphism} implies that the induced $*$-homomorphism $j_A \times j_X \colon \cO_X \to \cO_{X\oplus Y}$ is gauge-equivariant and injective.
\end{proof}

\begin{lemma}\label{lem:reducing_covariance_ideals}
	If $(\phi_X,{}_A X_B)$ and $(\phi_Y, {}_B Y_C )$ are $C^*$-correspondences, then $J_{\phi_{X \ox Y}} \subseteq J_{\phi_X}$. 
\end{lemma}
\begin{proof}
	It follows from \cite[Corollary 3.7]{Pimsner} that if $\phi_X(a) \ox \Id_Y \in \End_C^0(X \ox Y)$, then $\phi_X(a) \in \End_A^0(X \cdot \phi_Y^{-1} (\End_B^0(Y)))$.  It is clear that $\ker(\phi_X) \subseteq \ker(\phi_X \ox \Id_{Y})$ so $\ker(\phi_X \ox \Id_Y)^{\perp} \subseteq \ker(\phi_X)^{\perp}$. The result now follows. 
\end{proof}

Let $(\phi_X, {}_A X_A)$ be a correspondence and $N$ a positive integer. 
With $\phi_{X^{\ox N}} \coloneqq \phi_X\ox \Id_{N-1}$ the pair $(\phi_{X^{\ox N}}, X^{\ox N})$ is a correspondence over $A$.
Given a representation $(\alpha,\beta) \colon (\phi_X, {}_A X_A) \to B$ 
we denote by $\beta^{\ox N} \colon X^{\ox N} \to B$ the induced map $\beta^{\ox N} (x_1 \ox \cdots \ox x_N) = \beta(x_1) \cdots \beta(x_N)$, for all $x_1 \ox \cdots \ox x_N\in X^{\ox N}$ 
and note that $(\alpha,\beta^{\ox N})$ is a representation of $(\phi_{X^{\ox N}}, X^{\ox N})$ in $B$, \cite[\S 2]{Katsura2004-jfa}. We let $\beta^{(N)} \coloneqq (\beta^{\ox N})^{(1)} \colon \End^0_A(X^{\ox N}) \to B$. 

Our next lemma records how the Cuntz--Pimsner algebra of $(\phi_{X^{\ox N}}, X^{\ox N})$ embeds in the Cuntz--Pimsner algebra of $(\phi_{X}, X_A)$. 
However, we note that this embedding is not gauge-equivariant in the usual sense. 

\begin{lemma}\label{lem:covariant_powers}
  Let $(\phi_X, {}_A X_A)$ be a nondegenerate correspondence
  and let $(\iota_A,\iota_X)\colon (\phi_X, {}_A X_A) \to \cO_X$ be a universal representation.
  Then for each $N \in \N$, $(\iota_A,\iota_X^{\ox N}) \colon (\phi_{X^{\ox N}}, {}_A X_A^{\ox N}) \to \cO_X$ is an injective covariant representation.
  In particular, there is an induced injective $*$-homomorphism $\tau_N \colon \cO_{X^{\ox N}} \to \cO_{X}$ such that $\tau_N \circ \iota_{X^{\ox N}} = \iota_X^{\ox N}$. 
  
  Furthermore, let $\gamma \colon \bT\to \Aut(\cO_{X})$ denote the gauge action on $\cO_X$ and let $\ol{\gamma} \colon \bT \to \Aut(\cO_{X^{\ox N}})$ denote the gauge action on $\cO_{X^{\ox N}}$. Consider the $N$-th power $\ol{\gamma}^N \colon \bT \to \Aut(\cO_{X^{\ox N}})$ of the gauge action on $\cO_{X^{\ox N}}$: so $\ol{\gamma}^N_z(\iota_A(a)) = \iota_A(a)$ and $\ol{\gamma}^N_z(\iota_X(x)) = z^{N} \iota_X(x)$ for all $a \in A$ and $x \in X^{\ox N}$. Then $\tau_N \circ \ol{\gamma}^N_z = \gamma_z \circ \tau_N$ for all $z \in \bT$. 
\end{lemma}

\begin{proof}
  We need to verify that $(\iota_A,\iota_X^{\ox N})$ is covariant.	  
 It follows from \cref{lem:reducing_covariance_ideals} that $J_{\phi_{X^{\ox N}}} \subseteq J_{\phi_X}$. Recall that a rank-1 operator in $\End^{0}_A(X \cdot J_{\phi_{X^{\ox N}}})$ may be written in the form $\Theta_{x \cdot a, y}$ for $x,y \in X_A$ and $a \in J_{\phi_{X^{\ox N}}}$. Then $\Theta_{x \cdot a, y} \ox \Id_{N-1} \in \End_A^0(X^{\ox N})$ since $J_{\phi_{X^{\ox N}}} \subseteq J_{\phi_{X^{\ox N-1}}}$. Moreover, if  $(x_i)$ is a frame for $X^{\ox (N-1)}_A$, then
\[
\Theta_{x \cdot a, y} \ox \Id_{N-1} =  \sum_{i} \Theta_{x \ox \phi(a)x_i, y \ox x_i}.
\]
We proceed by induction, the base case being covariance of $(\iota_A,\iota_X)$, which is given. 
Suppose for induction that $(\iota_A, \iota_X^{\ox N-1})$ is covariant. 
This is equivalent to the fact that $\iota_A(a) =  \sum_i \iota_X^{(N-1)} (\Theta_{\phi_X(a) e_i,e_i})$ for all $a \in J_{X^{\ox (N-1)}}$. (cf. \cite[Remark 3.9]{Pimsner}). 
Using the inductive hypothesis at the second to last equality, it follows that for any $x,y \in X_A$ and $a \in J_{\phi_{X^{\ox N}}}$, 
\begin{align*}
\iota_X^{(N)} (\Theta_{x \cdot a, y} \ox \Id_{N-1}) 
	&= \iota_X^{(N)} \Big(\sum_i  \Theta_{x \ox \phi(a) x_i,y \ox x_i}\Big)
	= \sum_i \iota_X(x)\iota_X^{\ox {(N-1)}}(\phi(a) x_i) \iota_X^{\ox {(N-1)}}(x_i)^* \iota_X(y)^*\\
	&=  \iota_X(x) \iota_X^{(N-1)}\Big(\sum_i\Theta_{\phi(a)x_i,x_i}\Big) \iota_X(y)^*
	= \iota_X(x) \iota_A(a) \iota_X(y)^*
	= \iota_X^{(1)} (\Theta_{x \cdot a, y}). 
\end{align*}
It follows that for any $T \in \End_A^0(X \cdot J_{\phi_{X^{\ox N}}})$ we have $\iota_X^{(1)}(T) = \iota_X^{(N)}(T \ox \Id_{N-1})$.
Covariance of $(\iota_A,\iota_X)$ now implies that for all $a \in J_{\phi_{X^{\ox N}}}$,
\begin{align*}
	\iota_A(a) = \iota^{(1)}_X(\phi_X(a)) = \iota^{(N)}_X(\phi_X(a) \ox \Id_{N-1})
\end{align*}
so that $(\iota_A,\iota_X^{\ox N})$ is covariant. 
The universal property of $\cO_{X^{\ox N}}$ yields a $*$-homomorphism $\tau_N \colon \cO_{X^{\ox N}} \to \cO_{X}$ 
satisfying $\tau_N \circ \iota_A = \iota_A$ and $\tau_N \circ \iota_{X^{\ox N}} = \iota_{X}^{\ox N}$.

By considering local $N$-th roots, the fixed point algebras $\cO_{X^{\ox N}}^{\ol{\gamma}}$ and $\cO_{X^{\ox N}}^{\ol{\gamma}^N}$ can be seen to coincide. 
Moreover, it is straightforward to see that $\tau_N \circ \ol{\gamma}^N_z = \gamma_z \circ \tau_N$ for all $z \in \bT$. 
With minimal adjustments, the proof of the Gauge-Invariant Uniqueness Theorem found in \cite[Theorem~6.4]{Katsura2004-jfa} carries over to the action $\ol{\gamma}^N$, 
so since $(\iota_A, \iota_X^{\ox N})$ is an injective representation it follows that $\tau_N$ is injective.
\end{proof}

The next theorem is the main result of~\cite{MPT}: strong shift equivalent $C^*$-correspondences (that are nondegenerate and regular) have Morita equivalent Cuntz--Pimsner algebras.
Here we simply sketch the proof to make it clear that the Morita equivalence Muhly, Pask, and Tomforde construct
in fact implements a gauge-equivariant Morita equivalence.
This is certainly known (or at least anticipated) by experts but we consider it worthwhile to mention it.

\begin{thm}[{\cite[Theorem 3.14]{MPT}}] \label{thm:sse-equivariant-morita-equivalence}
  Suppose $(\phi_X, {}_A X_A)$ and $(\phi_Y, {}_B Y_B)$ are nondegenerate and regular correspondences.
  If they are strong shift equivalent, then the Cuntz--Pimsner algebras $\cO_X$ and $\cO_Y$ are gauge-equivariantly Morita equivalent.
\end{thm}

\begin{proof}
It suffices to assume that $X_A$ and $Y_B$ are elementary strong shift equivalent.
Choose nondegenerate and regular correspondences $(\phi_R, {}_A R_B)$ and $(\phi_S, {}_BS_A)$ (cf.~\cite[Section 3]{MPT}) such that
\[
  X_A \cong R\ox_B S \quad \textrm{and} \quad Y_B \cong S\ox_A R.
\]
By \cref{lem:covariant_sums} we have covariant morphisms $(j_A,j_X) \colon (\phi_X,{}_AX_A) \to (\phi_{X \oplus Y}, {}_{A \oplus B} X \oplus Y_{A\oplus B})$ and $(j_B, j_Y) \colon (\phi_Y, {}_BY_B) \to (\phi_{X \oplus Y}, {}_{A \oplus B} X \oplus Y_{A\oplus B})$.

Let $Z = S \oplus R$ be the correspondence over $A\oplus B$ with the obvious right module structure and left action
$\phi_Z\colon A\oplus B \to \End_{A\oplus B}(Z)$ given by $\phi_Z(a,b)(s,r) = (\phi_S(b)s, \phi_R(a)r)$ for all $(a,b)\in A\oplus B$ and $(r,s)\in Z$.\footnote{Muhly, Pask, and Tomforde call $Z$ the \emph{bipartite inflation}.} Then $Z^{\ox 2}$ is isomorphic to $X \oplus Y$ as $A \oplus B$-correspondences by \cite[Proposition 3.4]{MPT}.

By~\cref{lem:covariant_sums,lem:covariant_powers} there are inclusions $\lambda_X \colon \cO_X \to \cO_Z$ and $\lambda_Y \colon \cO_Y \to \cO_Z$ such that the diagram
\[\begin{tikzcd}[ampersand replacement=\&]
	{\cO_X} \\
	\& {\cO_{X \oplus Y} \cong \cO_{Z^{\ox 2}}} \& {\cO_{Z}} \\
	{\cO_Y}
	\arrow["{\tau_2}", from=2-2, to=2-3]
	\arrow["{j_B \times j_Y}", from=3-1, to=2-2]
	\arrow["{j_A \times j_X}"', from=1-1, to=2-2]
	\arrow["{\lambda_X}"', bend left = 20, from=1-1, to=2-3]
	\arrow["{\lambda_Y}", bend right = 20, from=3-1, to=2-3]
\end{tikzcd}\]
commutes.

As in~\cite[Lemma 3.12]{MPT}, we may construct full complementary projections $P_X$ and $P_Y$ in the multiplier algebra $\Mult(\cO_Z)$ (using approximate identities in $A$ and $B$, respectively)
such that $\lambda_X(\cO_X) = P_X \cO_Z P_X$ and $\lambda_Y(\cO_Y) = P_Y \cO_Z P_Y$ are full and $P_X + P_Y = 1_{\Mult(\cO_Z)}$. 

For gauge equivariant Morita equivalence (see \cref{def:gauge-ME}) we will produce a circle action on $\cO_Z$ which restricts to to the gauge actions on $\cO_X$ and $\cO_Y$. The action on $\cO_Z$ will not be the gauge action, as the gauge action on $\cO_Z$ runs at `half-speed' compared to the gauge actions on $\cO_X$ and $\cO_Y$.

Define an action of $\bT$ on $Z=S\oplus R$  by
\[
U_z(s,r)=(s,zr),\qquad (s,r)\in Z,\ z\in \bT.
\]
Conjugation by the second quantisation of $U_z$ on the Fock module
of $Z$ gives an action on the Toeplitz algebra of $Z$, which descends to $\cO_Z$, \cite{LN} and \cref{lem:LN}.

Since $X=R\ox_BS$ and $Y=S\ox_AR$, we see that the induced action on 
$Z^{\ox 2}\cong X\oplus Y$ is the sum of the actions of $\bT$ on $X$ and $Y$
given on $x\in X$ and $y\in Y$ by
\[
x\mapsto zx,\qquad y\mapsto zy,\qquad z\in\mathbb{T}.
\]
These actions induce the gauge actions of $\cO_X$ and $\cO_Y$, respectively.
\end{proof}

\begin{rmk}
Regularity of correspondences is not required for \cref{lem:covariant_sums,lem:covariant_powers}. We note however, that regularity plays a crucial role in the proof of \cref{thm:sse-equivariant-morita-equivalence}, namely in constructing the projections $P_X$ and $P_Y$. There are counterexamples to \cref{thm:sse-equivariant-morita-equivalence} when either $X$ or $Y$ is not regular (see \cite{MPT}).
\end{rmk}

\begin{rmk}\label{rmk:BDR}
In our proof above, it is important that the correspondences are regular.
It is however possible to have strong shift equivalent regular correspondences where the intermediate correspondences are \emph{not} regular, 
and this is not addressed in our approach.
This gap is however rectified by recent work of Bilich, Dor-On, and Ruiz \cite{Bilich-DorOn-Ruiz}. 
They show that even if two regular correspondences are strong shift equivalence in a way where the intermediate correspondences are not regular, it is possible to find different correspondences that are regular and that implement a strong shift equivalence, see Theorem 4.15 and the comment after its proof in \cite{Bilich-DorOn-Ruiz}.
We thank Adam Dor-On for bringing our attention to this gap (and its fix).
\end{rmk}

\section{In-splits}
\label{sec:insplit}

In this section, we recall the notion of in-splits for directed graphs, and extend the notion to both topological graphs and $C^*$-correspondences. 

\subsection{In-splits for topological graphs}

Let us start by recalling the classical notion from symbolic dynamics of in-splittings.
Let $E = (E^0,E^1,r,s)$ be a countable discrete directed graph.
Fix a vertex $w \in E^0$ which is not a source (i.e. $r^{-1}(w) \neq \varnothing$) and let $\Pp = \{\Pp_i\}_{i=1}^n$ be a partition of $r^{-1}(w)$ into a finite number of nonempty sets 
such that at most one of the partition sets $\Pp_i$ is infinite. 

Following \cite[Section 5]{Bates-Pask}, the \emph{in-split graph of $E$ associated to $\Pp$} is the directed graph $E_r(\Pp)$ defined by
\begin{align*}
E_r^0(\Pp) &= \{ v_1 \mid v\in E^0, v \ne w \} \cup \{w_1,\ldots,w_n\},\\
E_r^1(\Pp) &= \{ e_1 \mid e \in E^1, s(e) \ne w \} \cup \{e_1, \ldots, e_n \mid e \in E^1, s(e) = w\},\\
r_{\Pp}(e_i) &= 
\begin{cases}
r(e)_1 &\text{if } r(e) \ne w\\
w_j & \text{if } r(e) = w \text{ and } e \in \Pp_j,
\end{cases} \\
s_{\Pp}(e_i) &= s(e)_i,
\end{align*}
for all $e_i\in E^1_r(\Pp)$.

\begin{rmk}
If $E$ is a finite graph with no sinks and no sources, then the bi-infinite paths on $E$ define a two-sided shift of finite type (an edge shift).
The in-split graph $E_r(\Pp)$ is again a finite graph with no sinks and no sources, and the pair of edge shifts are topologically conjugate.  
In fact, if $\textsf{A}$ and $\textsf{A}(\Pp)$ denote the adjacency matrices of $E$ and $E_r(\Pp)$, respectively,
then there are rectangular nonnegative integer matrices $\textsf{R}$ and $\textsf{S}$ such that $\textsf{A} = \textsf{RS}$ and $\textsf{SR} = \textsf{A}(\Pp)$.
That is, the matrices are strong shift equivalent, cf. \cite[Chapter 7]{Lind-Marcus}. 
\end{rmk}

\begin{example}\label{ex:in-split}
Consider the directed graphs
\[
\begin{tikzpicture}
[baseline=-0.25ex,
vertex/.style={
circle,
fill=black,
inner sep=1.5pt
},
edges/.style={
-stealth,
shorten >= 3pt,
shorten <= 3pt
},
scale =1]

\node[vertex] (a) at (0,0) {};%
\node[vertex] (b) at (2,0) {};%

\node[anchor= east] at (a) {\scriptsize{$w$}};
\node[anchor= west] at (b) {\scriptsize{$v$}};

\draw[edges] (b.west) -- node[anchor=south, inner sep = 2pt]{\scriptsize{$g$}} (a.east);

\draw[edges] (b.west) to [out = 225, in = -45] node[anchor=south, inner sep = 2pt]{\scriptsize{$h$}} (a.east);

\draw[edges,blue] (a.east) to [out = 45, in = -225] node[anchor=south, inner sep = 2pt]{\scriptsize{$f$}} (b.west);

\draw[edges,red] (a.north) to [out = 45, in = -225, min distance = 40pt, looseness = 10] node[anchor=south, inner sep = 2pt]{\scriptsize{$e$}} (a.north);
\end{tikzpicture}
\qquad
\text{and}
\qquad 
\begin{tikzpicture}
[baseline=-5ex,
vertex/.style={
circle,
fill=black,
inner sep=1.5pt
},
edges/.style={
-stealth,
shorten >= 3pt,
shorten <= 3pt
},
scale =1]

\node[vertex] (a) at (0,0) {};%
\node[vertex] (b) at (2,0) {};%
\node[vertex] (c) at (0,-2) {};%

\node[anchor= east] at (a) {\scriptsize{$w_1$}};
\node[anchor= west] at (b) {\scriptsize{$v_1$}};
\node[anchor= east] at (c) {\scriptsize{$w_2$}};

\draw[edges] (b.south west) -- node[anchor=south east, inner sep = 1pt]{\scriptsize{$g_1$}} (c.north east);

\draw[edges] (b.west) to node[anchor=south, inner sep = 2pt]{\scriptsize{$h_1$}} (a.east);

\draw[edges,blue] (a.east) to [out = 45, in = -225] node[anchor=south, inner sep = 2pt]{\scriptsize{$f_1$}} (b.west);
\draw[edges,blue] (c.north east) to [out = 0, in = 270] node[anchor=south east, inner sep = 1pt]{\scriptsize{$f_2$}} (b.west);

\draw[edges,red] (a.north) to [out = 45, in = -225, min distance = 40pt, looseness = 10] node[anchor=south, inner sep = 2pt]{\scriptsize{$e_1$}} (a.north);
\draw[edges,red] (c.north) to node[anchor=east, inner sep = 2pt]{\scriptsize{$e_2$}} (a.south);
\end{tikzpicture}.
\]
Note that the loop $e$ is both an incoming and an outgoing edge for $w$.
Partition $r^{-1}(w)$ into $\Pp = \{\Pp_1,\Pp_2\}$ with $\Pp_1 = \{ e,h\}$ and $\Pp_2 = \{g\}$. 
The right-most graph above is then the in-split graph with respect to $\Pp$. The outgoing edges from $w$ are coloured to highlight their duplication in the in-split graph. 

The adjacency matrices of the graph and its in-split are
\[
\textsf{A} = 
\begin{pmatrix}
1 & 1 \\
2 & 0
\end{pmatrix} \qquad \textrm{and} \qquad
\textsf{B} = 
\begin{pmatrix}
1 & 0 & 1 \\
1 & 0 & 1 \\
1 & 1 & 0
\end{pmatrix},
\]
respectively, and the rectangular matrices 
\[
\textsf{R} = 
\begin{pmatrix}
1 & 0 \\
1 & 0 \\
1 & 1
\end{pmatrix} \qquad \textrm{and} \qquad
\textsf{S} = 
\begin{pmatrix}
1 & 0 & 1 \\
0 & 1 & 0
\end{pmatrix}
\]
satisfy $\mathsf{B = RS}$ and $\mathsf{SR = A}$.
This is an example of an (elementary) strong shift equivalence.
\end{example}

Suppose $E$ is a graph and let $E(\Pp)$ be an in-split graph.
Define a finite-to-one surjection $\alpha \colon E^0_r(\Pp) \to E^0$ by $
\alpha(v_i) = v
$
for all $v\in E_r^0(\Pp)$  (forgetting the labels)
and use the partition to define a map $\psi \colon E^1 \to E_r^0$ by
\[
\psi(e) = 
\begin{cases}
r(e)_1 & \text{if } r(e) \ne w,\\
w_i & \text{if } r(e) = w \text{ and } e \in \Pp_i,
\end{cases}
\]
for all $e\in E^1$.
Note that since $w$ is not a source, $\alpha$ maps sources bijectively to sources, 
and since at most one set in $\Pp$ contains infinitely many edges, it also follows that $\alpha$ maps infinite receivers bijectively to infinite receivers.

Our first observation is that $r = \alpha \circ \psi$,  so that an in-split may be thought of as a factorisation of the range map $r \colon E^1 \to E^0$ through the new vertex set $E^0_r(\Pp)$.
For our second observation, consider the fibred product 
\[
E^1_r \coloneqq E^1 \times_{s,\alpha} E^0(\Pp) = \{ (e,v_i)\in E^1\times E^0(\Pp) : s(e) = v_i \}
\]
The map from $E^1(\Pp)$ to $E^1_r$ given by $e_i\mapsto(e,s(e)_i)$ 
induces a graph isomorphism between $E(\Pp)$ and $E_r$. 
These observations allow us to define in-splits for more general topological graphs.

\begin{defn}
\label{defn:topological_graph_insplit}
An \emph{in-split} (or \emph{range-split}) of a topological graph $E = (E^0,E^1,r,s)$ is a triple $I = (\alpha,E^0_I,\psi)$ consisting of
\begin{enumerate}
\item a locally compact Hausdorff space $E^0_I$,
\item a continuous map $\psi \colon E^1 \to E^0_I$, and
\item a continuous and proper surjection $\alpha \colon E^0_I \to E^0$ that restricts to a homeomorphism between $E^0_{I,\psi-\sing}$ and $E^0_{\sing}$,
\end{enumerate}
such that $\alpha \circ \psi = r$.
\end{defn}

\begin{rmk}
For directed graphs the continuity assumptions of an in-split $I = (\alpha, E^0_I,\psi)$ are automatic. 
The properness of $\alpha$ can be reinterpreted as requiring that $|\alpha^{-1}(v)| < \infty$ for all $v \in E^0$. 
In the case of directed graphs the notion of in-split introduced in \Cref{defn:topological_graph_insplit} directly generalises that of \cite[Section 5]{Bates-Pask} 
(with source and range maps flipped).
\end{rmk}

We associate a new topological graph to an in-split.

\begin{lemma} 
Let $E = (E^0,E^1,r,s)$ be a topological graph and let $I = (\alpha, Y,\psi)$ be an in-split of $E$.
Then $E_I = (E_I^0, E_I^1,r_I,s_I)$ is a topological graph, where 
\begin{enumerate}
\item $E^1_I \coloneqq E^1 \times_{s,\alpha} E^0_I = \{ (e,v) \in E^1 \times E^0_I \mid s(e) = \alpha(v)\}$ equipped with the subspace topology of the product $E^1 \times E^0_I$; and
\item $r_I(e,v) = \psi(e)$ and $s_I(e,v) = v$, for all $e\in E^1$ and $v\in E_I^0$.
\end{enumerate}
Moreover, $E^0_{I,r_I-\sing}$ and $E^0_{\sing}$ are homeomorphic via $\alpha$.
\end{lemma}

\begin{proof}
The space $E_I^1$ is locally compact as a closed subspace of $E^1 \times E^0_I$ and the maps $r_I$ and $s_I$ are clearly continuous.
To see that $s_I$ is open, take open sets $U$ in $E^1$ and $V$ in $E_I^0$
and consider the basic open set $W = (U \times V) \cap E_I^1$ in $E^1_I$.
Then $s_I(W) = \alpha^{-1}(s(U)) \cap V$ which is open in $E^0_I$, so $s_I$ is open.

To see that $s_I$ is locally injective, fix $(e,v) \in E_I^1$.
Since $s$ is locally injective, there exists an open neighbourhood $U$ of $e$ in $E^1$ such that $s|_U$ is injective. 
Let $V$ be any open neighbourhood of $v$ in $E^0_I$. 
Then $W = (U \times V) \cap E_I^1$ is an open neighbourhood of $(e,v)$ in $E^1_I$.
If $(e',v'),(e'',v'') \in W$ are such that $v' = s_I(e',v') = s_I(e'',v'') = v''$, then $s(e') = \alpha(v') = s(e'')$ so that $e' = e''$. 
We conclude that $s_I$ is a local homeomorphism and so $E_I$ is a topological graph.

For the final statement we show that the $r_I$-singular and $\psi$-singular subsets of $E^0_I$ coincide, 
and then appeal to the fact that $\alpha$ restricts to a homeomorphism between $E^0_{I,\psi-\sing}$ and $E^0_{\sing}$. 
First observe that since $\alpha$ is surjective, we have $r_I(E_I^1) = \psi(E^1)$ and so $E^0_{I, r_I-\src} = E^{0}_{I,\psi-\src}$.

Now fix a precompact open set $V \subseteq E^0_I$, and observe that
\[
r_I^{-1}(\overline{V}) = \{ (e,v) \in E^1_I \mid \psi(e) \in \overline{V} \}.
\]
First suppose that $r_I^{-1}(\overline{V})$ is compact. 
Let $p_1 \colon E^1 \times_{s,\alpha}E^0 \to E^1$ denote the projection onto the first factor 
and observe that $p_1(r_I^{-1}(\overline{V})) = \psi^{-1}(\overline{V})$ since $\alpha$ is surjective.
Moreover, the set is compact since $p_1$ is continuous, so $E^0_{I,r_I-\fin} \subseteq E^0_{I,\psi-\fin}$. 

Now suppose that $\psi^{-1}(\overline{V})$ is compact. 
Since $\alpha$ is proper and $s$ is continuous, $\alpha^{-1}(s(\psi^{-1}(\overline{V})))$ is compact in $E^0_I$.
Since $E^0$ is Hausdorff, $E_I^1$ is a closed subspace of $E^1 \times E^0_I$. 
Consequently,
\begin{align*}
r_I^{-1}(\overline{V}) 
&= \psi^{-1}(\overline{V}) \times_{s,\alpha}  \alpha^{-1}(s(\psi^{-1}(\overline{V})))
= (\psi^{-1}(\overline{V}) \times  \alpha^{-1}(s(\psi^{-1}(\overline{V})))) \cap E^1_I
\end{align*}
is a closed subspace of the compact product $\psi^{-1}(\overline{V}) \times  \alpha^{-1}(s(\psi^{-1}(\overline{V})))$, and therefore compact. 
It follows that, $E^0_{I,\fin} = E^0_{I,\psi-\fin}$ and so $E^0_{I,\sing} = E^0_{I,\psi-\sing}$ as desired.
\end{proof}
\begin{rmk}\label{rmk:psi_regular_is_good}
Let $E$ be a regular topological graph (so $E^0_{\sing} = \varnothing$) and $I = (\alpha,E_I^0,\psi)$ an in-split of $E$.
The condition that $\alpha$ restricts to a homeomorphism on singular sets implies that $E^0_{I,\reg} = E^0_{I}$ so $E_I$ is also regular.
In particular, $\psi$ is  both proper and surjective in this case. 
\end{rmk}

\begin{defn}
We call $E_I = (E_I^0,E_I^1, r_I, s_I)$ the \emph{in-split graph of $E$ via $I$}.
\end{defn}

Williams'~\cite{Williams1973} original motivation for introducing state splittings---such as in-splits---was that even if the in-split graph is different, 
the dynamical systems they determine (the edge shifts) are topologically conjugate. 
We proceed to prove that this is also the case for our in-splits for topological graphs.
It is interesting to note that our approach provides a new proof of this fact even in the classical case of discrete graphs.
To do this we need some lemmas.

\begin{lemma}\label{lem:extending_alpha_to_edges}
Let $I = (\alpha,E^0_I,\psi)$ be an in-split of a topological graph $E = (E^0,E^1,r,s)$.
The projection onto the first factor $\alpha_1 \colon E^1_I = E^1 \times_{s,\alpha} E^0_I \to E^1$ is continuous, proper, and surjective. 
Moreover, the following diagram commutes:
\[
\begin{tikzcd}[ampersand replacement=\&]
{E^0_I} \& {E^1_I} \\
{E^0} \& {E^1}
\arrow["\alpha"', from=1-1, to=2-1]
\arrow["{\alpha_1}", from=1-2, to=2-2]
\arrow["{r_I}"', from=1-2, to=1-1]
\arrow["r", from=2-2, to=2-1]
\arrow["\psi"', from=2-2, to=1-1]
\end{tikzcd} .
\]
\end{lemma}
\begin{proof}
It is clear that $\alpha_1$ is continuous, and surjectivity follows from surjectivity of $\alpha$. 
If $K$ is a compact subset of $E^1$, then 
\[
\alpha_1^{-1}(K) = K \times_{s,\alpha} \alpha^{-1}(s(K)) = (K \times \alpha^{-1}(s(K))) \cap E^1_I.
\]
Since $\alpha$ is proper and $s$ continuous, $\alpha_1^{-1}(K)$ is a closed subset of the compact set $K \times \alpha^{-1}(s(K))$, so $\alpha_1$ is proper. 
Commutativity of the diagram follows from the definition of $r_I$. 
\end{proof}

We recall that the \emph{$n$-th power} of a topological graph $E$ is the topological graph $E^{(n)} \coloneqq (E^0,E^n,r,s)$ where $r(e_1\cdots e_n) \coloneqq r(e_1)$ and $s(e_1\cdots e_n) \coloneqq s(e_n)$.
We record how taking powers of topological graphs interacts with in-splits.

\begin{lemma}\label{lem:insplit_paths}
Let $E = (E^0,E^1,r,s)$ be a topological graph and $I =  (\alpha, E_I^0,\psi)$ an in-split of $E$.  Then $E_I^n \simeq E^n \times_{s,\alpha} E_0^I$ for all $n \ge 1$, where $s \colon E^n \to E^0$ is given by $s(e_1\cdots e_n) = s(e_n)$. 
Moreover, if $\psi^{(n)} \colon E^n \to E^0_I$ is the map defined by $\psi^{(n)}(e_1\cdots e_n) = \psi(e_1)$, 
then the $n$-th power graph $E_I^{(n)}$ be obtained from $E^{(n)}$ via the in-split $I^{(n)} = (\alpha, E^0_I, \psi^{(n)})$. 
\end{lemma}
\begin{proof}
First, observe that
\[
E^n_I = \{ (e_1,v_1,\ldots,e_n,v_n) \mid e_i \in E^1,\, v_i \in E^0_I,\,  s(e_i) = \alpha(v_i), \text{ and } v_i = \psi(e_{i+1}) \text{ for all } i \ge 1\}.
\]
Since $\alpha \circ \psi = r$ it follows that the map  $(e_1,v_1,\ldots,e_n,v_n) \mapsto (e_1\cdots e_n, v_n)$ from $E_I^n$ to $E^n \times_{s,\alpha} E^0_I$ is a homeomorphism with inverse $(e_1\cdots e_n,v_n) \mapsto (e_1, \psi(e_2), e_2, \ldots, \psi(e_n), e_n, v_n)$.
The final statement follows immediately. 
\end{proof}

We now show that in-splits of regular topological graphs induce  topological conjugacies. Recall that for a regular topological graph $E$ the \emph{infinite path space} is given by 
\[
E^{\infty} = \Big\{e_1e_2\ldots \in \prod_{i=1}^{\infty} E^1 \mid s(e_i) = r(e_{i+1})\Big\}
\]
with a cylinder set topology making it a locally compact Hausdorff space. The \emph{shift map} $\sigma_E \colon E^{\infty} \to E^{\infty}$ is the local homeomorphism defined by $\sigma_E(e_1e_2\ldots) = e_2e_3 \ldots$. 

\begin{thm}\label{thm:top_insplits_conjugate}
Let $E = (E^0,E^1,r,s)$ be a regular topological graph and let $I = (\alpha, E^0_I,\psi)$ be an in-split of $E$.
Then the dynamical systems on the infinite path spaces $(\sigma_E, E^\infty)$ and $(\sigma_{E_I},E_I^{\infty})$ are topologically conjugate. 
\end{thm}

\begin{proof} 
Use \cref{lem:insplit_paths} to identify $E_I^n$ with $E^n \times_{s,\alpha} E^0_I$.  For each $n \ge 1$ let $r^n \colon E^{n+1} \to E^n$ be the map given by $r^n(e_1\cdots e_{n+1}) = e_1\cdots e_n$. Then $r_I^n \colon E^{n+1}_I \to E^n$ satisfies $r_I^n (e_1 \cdots e_{n+1}, v_{n}) = (e_1\cdots e_n,\psi(e_{n+1}))$. Define $\psi^n \colon E^{n+1} \to E_I^{n}$ by
$\psi^n(e_1\ldots e_{n+1}) = (e_1\cdots e_n, \psi(e_{n+1}))$,
and let $\alpha^n \colon E^n_I \to E^n$ be the projection onto the first factor. 
It is then routine to verify that the diagram	
\begin{equation}\label{eq:insplit_proj_lim}
\begin{tikzcd}[ampersand replacement=\&,column sep = 40pt]
{E^0_I} \& {E^1_I} \& {E^2_I} \& \cdots \& {E^{\infty}_I} \\
{E^0} \& {E^1} \& {E^2} \& \cdots \& {E^{\infty}}
\arrow["{r_I}"', from=1-2, to=1-1]
\arrow["{r}", from=2-2, to=2-1]
\arrow["{r^1}", from=2-3, to=2-2]
\arrow["{r^2_I}"', from=1-4, to=1-3]
\arrow["{r^2}", from=2-4, to=2-3]
\arrow["{\alpha}"', from=1-1, to=2-1]
\arrow["{\psi}"', from=2-2, to=1-1]
\arrow["{\alpha^1}"', from=1-2, to=2-2]
\arrow["{\psi^1}"', from=2-3, to=1-2]
\arrow["{r^1_I}"', from=1-3, to=1-2]
\arrow["{\alpha^2}"', from=1-3, to=2-3]
\arrow["{\alpha^{\infty}}"', shift right=1, from=1-5, to=2-5]
\arrow["{\psi^{\infty}}"', shift right=1, from=2-5, to=1-5]
\arrow[shift left=1, from=2-5, to=2-4]
\arrow[shift left=1, from=1-5, to=1-4]
\arrow["{\psi^2}"', from=2-4, to=1-3]
\end{tikzcd}
\end{equation}
commutes, 
where $\alpha^{\infty}$ and $\psi^{\infty}$ are induced by the universal properties of the projective limit spaces $E^\infty \simeq \varprojlim (E^n,r^n)$ and $E_I^{\infty} \simeq \varprojlim (E^n_I, r_I^n)$. 
In particular, $E_I^{\infty}$ and $E^{\infty}$ are homeomorphic via $\alpha^{\infty}$ and $\psi^{\infty}$. 

For conjugacy, we make the key observation that if $s^n \colon E^{n+1} \to E^n$ is given by $s^n(e_1\cdots e_{n+1}) = e_2\cdots e_{n+1}$, 
then the shift $\sigma_E \colon E^{\infty} \to E^{\infty}$ is the unique map making the diagram
\[\begin{tikzcd}[ampersand replacement=\&,column sep = 40pt]
{E^1} \& {E^2} \& {E^3} \& \cdots \& {E^{\infty}} \\
{E^0} \& {E^1} \& {E^2} \& \cdots \& {E^{\infty}}
\arrow["{r^1}"', from=1-2, to=1-1]
\arrow["{r}", from=2-2, to=2-1]
\arrow["{r^1}", from=2-3, to=2-2]
\arrow["{r^3}"', from=1-4, to=1-3]
\arrow["{r^2}", from=2-4, to=2-3]
\arrow["{s}"', from=1-1, to=2-1]
\arrow["{s^1}"', from=1-2, to=2-2]
\arrow["{r^2}"', from=1-3, to=1-2]
\arrow["{s^2}"', from=1-3, to=2-3]
\arrow["\sigma_E"', shift right=1, from=1-5, to=2-5]
\arrow[shift left=1, from=2-5, to=2-4]
\arrow[shift left=1, from=1-5, to=1-4]
\end{tikzcd}\]
commute. 
With a similar commuting diagram for the shift $\sigma_{E_I} \colon E^{\infty}_I \to E^{\infty}_I$, it follows from \labelcref{eq:insplit_proj_lim} that $\alpha^{\infty} \circ \sigma_{E_I} = \sigma_E \circ \alpha^{\infty}$. 
\end{proof}

\begin{rmk}
The condition that the topological graphs be regular should not be essential.
A similar argument---though more technically demanding---should work for general topological graphs by replacing the path space $E^{\infty}$ with the boundary path space and using the direct limit structure of the boundary path space outlined in either \cite[\S 3.3.1]{MundeyPhD} or \cite{Katsura2021}. 
\end{rmk}

\begin{rmk}
We have seen that any in-split induces a conjugacy of the limit dynamical systems.
In the case of shifts of finite type, this was first proved by Williams~\cite{Williams1973} where he also showed that the converse holds:
any conjugacy is a composition of particular conjugacies that are induced from in-splits and their inverses.
We do not know whether a similar result could hold in the case of topological graphs.
\end{rmk}

\begin{examples}
Let $E$ be a regular topological graph.
\begin{enumerate}
\item We refer to $I = (\Id_{E^0}, E^0, r)$ as the \emph{identity in-split} since $E_I$ is graph isomorphic to $E$.
\item We refer to $I = (r,E^1,\Id_{E^1})$ as the \emph{complete in-split} of $E$. The topological graph associated to $I$ 
has vertices $E_I^0 = E^1$ and edges 
\[
E_I^1 = E^1 \times_{s,r} E^1  = \{ (e',e) \in E^1\times E^1 : s(e') = r(e) \}
\]
that may be identified with $E^2$, the composable paths of length $2$.
The range and source maps satisfy $r_I(e',e) = \Id_{E^1}(e') = e'$ and $s_I(e',e) = e$, for all $(e',e)\in E_I^1$.
We denote this in-split graph by $\hat{E} = (E^1, E^2, \hat{r}, \hat{s})$ and refer to it as the \emph{dual graph} of $E$. 

When $E$ is a regular topological graph, then $E_I$  is graph isomorphic to Katsura's dual graph, cf.~\cite[Definition 2.3]{Katsura2021} (see also \cite[Remark 2.3]{Brenken}),
and when $E$ is discrete, then $E_I$ is discrete and it is graph isomorphic to the dual graph as in~\cite[Corollary 2.6]{Raeburn}.
Iterating the dual graph construction in the case of topological graphs coincides with Katsura's iterative process in~\cite[Section 3]{Katsura2021}.
\end{enumerate}
\end{examples}

The following lemma is akin to \cite[Lemma 2.4]{Boyle-Franks-Kitchens} (see also \cite{Williams1973}) in the setting of nonnegative integer matrices. 
This lemma shows that the dual graph is in some sense the ``largest'' in-split of a regular topological graph.

\begin{lemma}\label{lem:top_graph_diamond_lemma}
Let $E$ be a regular topological graph and let $I = (\alpha, E_I^0, \psi)$ be an in-split of $E$.
Let $\alpha_1 \colon E^1_I \to E^1$ be the projection onto the first factor as in~\cref{lem:extending_alpha_to_edges}.
Then $I' = (\psi, E^1, \alpha_1)$ is an in-split of $E_I$ with the property that $(E_I)_{I'}$ is graph isomorphic to the dual graph $\hat{E}$.
\end{lemma}

\begin{proof}
Since $E$ is regular, it follows from~\cref{rmk:psi_regular_is_good} that $\psi$ is proper and surjective, 
and \cref{lem:extending_alpha_to_edges} implies that $\alpha_1$ is proper and surjective. 
Therefore, $I' = (\psi, E^1, \alpha_1)$ is an in-split of $E_I$.
Let $F = (E_I)_{I'}$ be the resulting in-split graph and observe that $F^0 = E^1$.
Moreover, we have 
\[
F^1 = E_I^1 \times_{s_I, \psi} E^1 = \{ (e',x,e) \in E^1 \times E^0_I \times E^1 \mid s(e') = \alpha(x) \text{ and } x = \psi(e)\}
\]
with $r_F(e',x,e) = \alpha_1(e',x) = e'$ and $s_F(e',x,e) = e$ for all $(e',x,e)\in F^1$.

The map $F^1 \to \hat{E}^1$ sending $(e',x,e) \mapsto (e',e)$ is a homeomorphism which intertwines the range and source maps.
It is injective because $x\in E_I^0$ is uniquely determined by $e'$, 
and it is surjective since if $(e',e)\in \hat{E}^1$ are composable edges, 
then $x = \psi(e')$ satisfies $\alpha(x) = \alpha\circ \psi(e') = r(e') = s(e)$, so $(e', x, e)$ is mapped to $(e',e)$.
\end{proof}

A simple class of examples comes from ``topologically fattening'' the class of directed graphs.
\begin{example}
Let $E = (E^0,E^1,r,s)$ be a regular directed graph and fix a locally compact Hausdorff space $X$. 
Let $F^0 \coloneqq E^0 \times X$ and $F^1 \coloneqq E^1 \times X$ with the respective product topologies and define $r_F(e,x) = (r(e),x)$ and $s_F(e,x) = (s(e),x)$. 
Then $F = (F^0,F^1,r_F,s_F)$ is a topological graph. 

If $I = (\alpha, E^0_I ,\psi)$ is an in-split of $E$, then $I_X \coloneqq (\alpha \times \Id_X, E^0_I \times X, \psi \times \Id_X)$ is an in-split of $F$. 
It is straightforward to check that the associated topological graph $F_{I_X}$ is isomorphic to $(E_I^0 \times X, E_I^1 \times X, r_I \times \Id_X, s_I \times \Id_X)$.
\end{example}

In the setting of topological graphs there are also strictly more exotic examples than those obtained via fattening directed graphs.

\begin{example}
\label{eg:n-m-circles}

Fix $m,n \in \Z \setminus \{0\}$ and let $E^0 \coloneqq \bT$ and $E^1 \coloneqq \bT$. 
Define $r,s \colon E^1 \to E^0$ by $r(z) = z^m$ and $s(z) = z^n$. Then $E = (E^0,E^1,r,s)$ is a topological graph.
Suppose $a,b \in \Z$ satisfy $m = ab$. Define $\psi \colon E^1 \to \bT$ by $\psi(z) = z^a$ and $\alpha \colon \bT \to E^0$ by $\alpha(z) = z^b$. 
Since $r(z) = z^m = (z^{a})^b = \alpha \circ \psi(z)$, it follows that $I = (\alpha, \bT, \psi)$ is an in-split of $E$. 
The new edge set is
\[
E^1_I = \{ (z_1,z_2) \in \bT^2 \mid z_1^n = z_2^b\}.
\] 
We claim that $E^1_I$ is homeomorphic to a disjoint union of $\gcd(n,b)$ copies of $\bT$. 

Let $q_b,q_n$ be the unique integers such that $n = q_n \gcd(n,b)$ and $b = q_b \gcd(n,b)$, and note that $q_n$ and $q_b$ have no common factors. We also record that $q_n b = q_n q_b \gcd(n,b) = q_b n$.
For each $|b|$-th root of unity $\omega$ define $\pi_{\omega} \colon \bT \to E^1_I$ by
\[
\pi_{\omega}(z) = (z^{q_b}, \omega z^{q_n}).
\]
Suppose that $(z_1,z_2) \in E^1_I$ and let $z$ be a $|q_b|$-th root of $z_1$. 
Then $(z^{q_n})^b = (z^{q_b})^n = z_1^n = z_2^b$, so $(z_2/z^{q_n})^b = 1$. 
Hence, there is some $|b|$-th root of unity $\omega$ such that $z_2 = \omega z^{q_n}$. 
In particular, every $(z_1,z_2) \in E^1_I$ can be written in the form $(z^{q_b},\omega z^{q_n}) = \pi_{\omega}(z)$ for some $z \in \bT$ and some $|b|$-th root of unity $\omega$.

We claim that each $\pi_\omega$ is injective. Suppose that $\pi_{\omega}(z) = \pi_{\omega}(v)$ for some $z,v \in \bT$. Then $z^{q_b} = v^{q_b}$ and $z^{q_n} = v^{q_b}$. Consequently $z = \omega_0 v$ for some $\omega_0 \in \bT$ that is simultaneously a $|q_b|$-th and a $|q_n|$-th root of unity. Since $q_b$ and $q_n$ are coprime, we must have $\omega_0 = 1$, so $\pi_\omega$ is injective. 
Since each $\pi_\omega$ is a continuous injection from a compact space to a Hausdorff space, it follows that each $\pi_\omega$ is a homeomorphism onto its image.  

Fix a primitive $|b|$-th root of unity $\lambda$. We claim that $\pi_{\lambda^c}$ and $\pi_{\lambda^d}$ have the same image if and only if $c \equiv kn +d \pmod{|b|}$ for some $0 \le k < |q_b|$. If  $c \equiv kn +d \pmod{|b|}$, then $\lambda^c = \lambda^{kn + d}$. For all $z \in \bT$ we compute
\begin{align*}
\pi_{\lambda^c} (z^{\gcd(n,b)}) 
&= ((z^{\gcd(n,b)})^{q_b}, \lambda^c (z^{\gcd(n,b)})^{q_n}) 
= (z^b, \lambda^{kn + d} z^n)\\
&= ((\lambda^{k\gcd(n,b)}z)^{q_b}, \lambda^d (\lambda^{k\gcd(n,b)} z)^{q_n})
= \pi_{\lambda^d}((\lambda^k z)^{\gcd(n,b)}).
\end{align*}
Since $z \mapsto z^{\gcd(n,b)}$ and $z \mapsto (\lambda^k z)^{\gcd(n,b)}$ both surject onto $\bT$, it follows that $\pi_{\lambda^c}$ and $\pi_{\lambda^d}$ have the same image. 

Conversely, suppose that $\pi_{\lambda^c}(z) = \pi_{\lambda^d}(v)$ for some $z,v \in \bT$. Then $z^{q_b} = v^{q_b}$ and $\lambda^c z^{q_n} = \lambda^d v^{q_n}$. Since  $z^{q_b} = v^{q_b}$, there is an $|q_b|$-th root of unity $\lambda_0$ such that $z = \lambda_0 v$. Since $b = q_b \gcd(n,b)$ there exists $0 \le k < |q_b|$ such that $\lambda_0 = \lambda^{k\gcd(n,b)}$. It follows that
\[
\lambda^c z^{q_n} = \lambda^d v^{q_n} =\lambda^d (\lambda^{k\gcd(n,b)} z)^{q_n} = \lambda^{kn + d} z^{q_n}.
\]
Therefore, $\lambda^c = \lambda^{kn + d}$ so $c \equiv kn +d \pmod{|b|}$.
It follows that $E_I^1$ is a disjoint union of circles, but what remains is to count how many distinct circles it is composed of. 

Since the maps $\pi_{\lambda^c}$ and $\pi_{\lambda^d}$ have the same image if and only if $c \equiv kn + d \pmod{|b|}$ for some $k$, the number of circles is in bijection with the cosets of $\bZ_{|b|} / n \bZ_{|b|}$. To determine the number of cosets it suffices to determine the cardinality of $n\bZ_{|b|}$. Using B\'ezout's Lemma at the last equality we observe that 
\[
n\bZ_{|b|} = \{nc \in \bZ_{|b|} \mid c \in \bZ\} = \{nc + bd \in \bZ_{|b|}\mid c,d \in \bZ_{|b|}\} = \{k \gcd(n,b) \in \bZ_{|b|} \mid k \in \bZ_{|b|}\}.
\]
It follows that $\bZ_{|b|}/n\bZ_{|b|}$ contains $\gcd(n,b)$ cosets.

For an explicit identification of $E^1_I$ with the disjoint union of $\gcd(n,b)$ circles, fix a primitive $|b|$-th root of unity $\lambda$ and let $\pi \colon \{1 ,\ldots, \gcd(n,b)\} \times \bT \to E_I^1$ be the homeomorphism defined by $\pi(k,z) = \pi_{\lambda^k}(z) = (z^{q_b},\lambda^k z^{q_n})$. Under this identification, 
\[
r_I(k,z) = \psi(z^{q_b})  = z^{q_ba} = z^{m/\gcd(n,b)} \quad \text{and} \quad s_I(k,z) = \lambda^k z^{q_n} = \lambda^k z^{n/\gcd(n,b)}.
\]
Remarkably, the quite different topological graphs $E$ and $E_I$ induce topologically conjugate dynamics on their respective path spaces by \cref{thm:top_insplits_conjugate}. This is far from obvious.

By swapping the role of $b$ and $n$ above, we could alternatively let $\gamma \in \bT$ be a primitive $|n|$-th root of unity to see that $\pi' \colon \{1,\ldots,\gcd(n,b)\} \times \bT \to E_I^1$ defined by $\pi'(k,z) = (\gamma^k z^{q_b}, z^{q_n})$ is a homeomorphism. Identifying $E^1_I$ with the disjoint union of circles via $\pi'$, the range and source maps for $E_I$ satisfy
\[
r_I(k,z) = \psi(\gamma^k z^{q_b}) = \gamma^{ka} z^{q_b a} = \gamma^{ka} z^{m/\gcd(n,b)} \quad \text{and} \quad s_I(k,z) = z^{q_n} = z^{n/\gcd(n,b)}.
\]
\end{example}

In general, the na\"ive composition of in-splits cannot be realised as a single in-split. 
If one pays the penalty of passing to paths, then the following result provides a notion of composition of in-splits, highlighting the role of the projective limit decomposition of path spaces from \eqref{eq:insplit_proj_lim}.

\begin{prop}
Suppose that $E$ is a regular topological graph and that there is a finite sequence of in-splits $I_k = (\alpha_k, E_{I_{k}}, \psi_k)$ for $k = 1,\ldots,n$ such that
\begin{enumerate}
\item $I_1$ is an in-split of $E$, and
\item $I_k$ is an in-split of $E_{I_{k-1}}$ for $k \ge 2$.
\end{enumerate}

Then $E_{I_n}^{(n)} = (E^0_{I_n},E^n_{I_n},r,s)$ is isomorphic to the graph obtained by a single in-split $(\alpha, E^0_{I_n}, \psi)$ of $E^{(n)} = (E^0,E^n,r,s)$. Moreover, $(\sigma_{E}^n,E^{\infty})$ is topologically conjugate to $(\sigma_{E_{I_n}}^n, E_{I_n}^{\infty})$. 
\end{prop}
\begin{proof}
For each $0 \le p,k \le n$ let $\alpha^p_k \colon E_{I_k}^p \to E_{I_{k-1}}^p$ and $\psi^p_k \colon E_{I_{k-1}}^p \to E_{I_{k}}^{p-1}$ be the maps arising from the sequences defined in the proof of \cref{thm:top_insplits_conjugate}, where for consistency we take the convention that $E_{I_0} \coloneqq E$, $\alpha_k^0 \coloneqq \alpha_k$, and $\psi_k^0\coloneqq \psi_k$. In particular the diagram	
\[\begin{tikzcd}[ampersand replacement=\&,column sep=50pt, row sep = 30pt]
{E_{I_{n}}^0} \& {E^1_{I_{n}}} \& \cdots \& {E^{n-1}_{I_{n}}} \& {E^n_{I_n}} \\
{E^0_{I_{n-1}}} \& {E^1_{I_{n-1}}} \& \cdots \& {E_{I_{n-1}}^{n-1}} \& {E^n_{I_{n-1}}} \\
\vdots \& \vdots \& \ddots \& \vdots \& \vdots \\
{E_{I_1}^0} \& {E_{I_1}^1} \& \cdots \& {E^{n-1}_{I_1}} \& {E_{I_1}^{n}} \\
{E^0} \& {E^1} \& \cdots \& {E^{n-1}} \& {E^{n}}
\arrow["{\alpha_1}"', from=4-1, to=5-1]
\arrow["{\psi_1}"{description}, from=5-2, to=4-1]
\arrow["r"', from=5-2, to=5-1]
\arrow["{\alpha_2}"', from=3-1, to=4-1]
\arrow["{\alpha_1^{n}}"', from=4-5, to=5-5]
\arrow["{\alpha_1^1}"', from=4-2, to=5-2]
\arrow["{\psi_2}"{description}, from=4-2, to=3-1]
\arrow["{\alpha_2^1}"', from=3-2, to=4-2]
\arrow["{r_{I_1}}"', from=4-2, to=4-1]
\arrow["{\alpha_2^n}"', from=3-5, to=4-5]
\arrow["{\psi_2^{n-1}}"{description}, from=4-5, to=3-4]
\arrow["{r^{n-1}}"', from=5-5, to=5-4]
\arrow["{r^{n-1}_{I_1}}"', from=4-5, to=4-4]
\arrow["{\psi_{1}^{n-1}}"{description}, from=5-5, to=4-4]
\arrow["{r^1}"', from=5-3, to=5-2]
\arrow["{r^1_{I_1}}"', from=4-3, to=4-2]
\arrow["{r_{I_1}^{n-2}}"', from=4-4, to=4-3]
\arrow["{r^{n-2}}"', from=5-4, to=5-3]
\arrow["{\alpha_1^{n-1}}"', from=4-4, to=5-4]
\arrow["{\alpha_2^{n-1}}"', from=3-4, to=4-4]
\arrow["{\psi_1^{n-2}}"{description}, from=5-4, to=4-3]
\arrow["{\psi^{n-2}_2}"{description}, from=4-4, to=3-3]
\arrow["{\psi_1^1}"{description}, from=5-3, to=4-2]
\arrow["{\alpha_n}"', from=1-1, to=2-1]
\arrow["{\alpha_{n-1}}"', from=2-1, to=3-1]
\arrow["{\psi_{n-1}}"{description}, from=3-2, to=2-1]
\arrow["{\psi_{n}}"{description}, from=2-2, to=1-1]
\arrow["{r_{I_{n-1}}}"', from=2-2, to=2-1]
\arrow["{r_{I_n}}"', from=1-2, to=1-1]
\arrow["{\psi_{n-1}^1}"{description}, from=3-3, to=2-2]
\arrow["{\alpha_{n-1}^{n-1}}"', from=2-4, to=3-4]
\arrow["{\alpha_{n-1}^n}"', from=2-5, to=3-5]
\arrow["{\alpha_n^n}"', from=1-5, to=2-5]
\arrow["{\alpha_n^{n-1}}"', from=1-4, to=2-4]
\arrow["{r_{I_n}^{n-2}}"', from=1-4, to=1-3]
\arrow["{r_{I_{n-1}}^1}"', from=2-3, to=2-2]
\arrow["{\psi_{n}^1}"{description}, from=2-3, to=1-2]
\arrow["{r_{I_n}^1}"', from=1-3, to=1-2]
\arrow["{r_{I_{n-1}}^{n-2}}"{description}, from=2-4, to=2-3]
\arrow["{\psi_n^{n-2}}"{description}, from=2-4, to=1-3]
\arrow["{r_{I_n}^{n-1}}"', from=1-5, to=1-4]
\arrow["{r_{I_{n-1}}^{n-1}}"{description}, from=2-5, to=2-4]
\arrow["{\psi_n^{n-1}}"{description}, from=2-5, to=1-4]
\arrow["{\psi_{n-1}^{n-1}}"{description}, from=3-5, to=2-4]
\arrow["{\alpha_n^1}"', from=1-2, to=2-2]
\arrow["{\alpha_{n-1}^1}"', from=2-2, to=3-2]
\end{tikzcd}\]
commutes.

Let $\alpha = \alpha_1 \circ \cdots \circ \alpha_n$ and let
$\psi = \psi_n^0 \circ \psi_{n-1}^1 \circ \cdots \circ \psi_2^{n-1} \circ \psi_1^n$.	We claim that $(\alpha, E^0_{I_n}, \psi)$ is an in-split of $E^{(n)}$. Clearly $\alpha$ is a continuous proper surjection and $\psi$ is continuous. Moreover, $\alpha \circ \psi = r \circ r^1 \circ \cdots r^{n-2} \circ r^{n-1}$ is the range map on the $n$-th power $E^{(n)}$. 

Repeatedly applying \cref{lem:insplit_paths} and using the fact that each $\alpha_i$ surjects, we see that
\begin{align*}
E_{I_n}^{n} 
\simeq E_{I_{n-1}}^n \times_{s,\alpha_n} E^0_{I_{n}} 
&\simeq (\cdots((E^n \times_{s,\alpha_1} E^0_{I_1}) \times_{s,\alpha_2} E^0_{I_2})\times_{s,\alpha_3} \cdots ) \times_{s,\alpha_n} E_{I_n}^0\\
&\simeq E^n \times_{s, \alpha_1 \circ \cdots \circ \alpha_n} E_{I_n}^0
= E^n \times_{s, \alpha} E_{I_n}^0.
\end{align*}
The source maps on $E^n \times_{s, \alpha} E_{I_n}^0$ as an in-split and $E_{I_n}^{n}$ clearly agree, and commutativity of the preceding diagram also imply that the range maps agree. 

The final statement follows after observing that $E^{\infty} \simeq \varprojlim (E^k,r^k) \simeq \varprojlim(E^{nk}, r^{nk})$ and applying \cref{thm:top_insplits_conjugate}.
\end{proof}

\subsection{Noncommutative in-splits}

Inspired by the recasting of in-splits for directed graphs and topological graphs 
we introduce the following analogous notion of in-splits for $C^*$-correspondences. 

\begin{defn}
{\label{defn:correspondence_insplit}}
An \emph{in-split} of a nondegenerate $A$--$A$-correspondence $(\phi, {}_AX_A)$ is a triple $I = (\alpha,B,\psi)$ consisting of a $C^*$-algebra $B$ together with a nondegenerate injective $*$-homomorphism  $\alpha \colon A \to B$ 
and a left action $\psi \colon B \to \End_A(X)$ such that $\psi \circ \alpha = \phi$ and, moreover, 
\begin{enumerate}
\item $\alpha(J_\phi) \subseteq J_{\psi} \coloneqq \psi^{-1}(\End^0_A(X)) \cap \ker(\psi)^{\perp}$, and
\item \label{itm:quotients_preserved} the induced $*$-homomorphism $\overline{\alpha}\colon A/J_{\phi} \to B/J_{\psi}$ is an isomorphism. 
\end{enumerate}
To an in-split $(\alpha,B,\psi)$ of $(\phi,X_A)$ we associate the $C^*$-correspondence $(\psi \ox \Id_B, X \ox_{\alpha} B)$ over $B$
where the left action is given as $(\psi\ox\Id_B)(b') (x\ox b) = \psi(b')x \ox b$ for all $x\in X_A$ and $b',b\in B$. We call this the \emph{in-split correspondence of $(\phi, {}_A X_A)$ associated to $I$}. 
\end{defn}
Observe that since $\phi$ and $\alpha$ are nondegenerate, so is $\psi$. We identify the covariance ideal for the in-split correspondence. 
\begin{lemma}
\label{lem:tee-ox1}
The ideal $J_{\psi}$ of $B$ is the covariance ideal for  $(\psi \ox \Id_B, X \ox_{\alpha} B)$. That is, $J_{\psi} = J_{\psi \ox \Id_B}$.
\end{lemma}
\begin{proof}
\cref{lem:reducing_covariance_ideals} implies $J_{\psi \ox \Id_B} \subseteq J_{\psi}$. 
For the other inclusion, observe that it follows from \cite[Corollary 3.7]{Pimsner} and 
\[
\alpha^{-1}(\End_B^0(B))=\alpha^{-1}(B)=A
\]
that the map $T \mapsto T \ox \Id_B$ from $\End_A(X)$ to $\End_B(X \ox_{\alpha} B)$ takes compact operators to compact operators. 
In particular, $\psi (b) \ox \Id_B$ is compact for each $b \in \psi^{-1}(\End_A^0(X))$, 
so $\psi^{-1}(\End_A^0(X)) \subset (\psi \ox \Id_B)^{-1}(\End_B^0(X \ox_{\alpha} B))$.

Clearly, $\ker(\psi) \subseteq \ker(\psi \ox \Id_B)$. 
On the other other hand, if $b_0 \in \ker(\psi \ox \Id_B)$, then
\[
0 = (\psi(b_0)x \ox b \mid \psi(b_0)x \ox b)_B = (b \mid \alpha((\psi(b_0)x, \psi(b_0)x)_A) b)_B
\]
for all $x \ox b \in X \ox_{\alpha} B$. 
In particular, $\alpha((\psi(b_0)x \mid \psi(b_0)x)_A) = 0$, so injectivity of $\alpha$ implies $\psi(b_0)x = 0$ for all $x \in X_A$. 
Hence, $\ker(\psi) = \ker(\psi \ox \Id_B)$.
We conclude that $J_{\psi \ox \Id_B} = J_{\psi}$. 	
\end{proof}

Condition~\ref{itm:quotients_preserved} allows for a useful decomposition of elements in $B$ in the following way.
\begin{lemma}\label{lem:b_decomposition}
For each $b \in B$ there exists $a \in A$ and $k \in J_{\psi}$ such that $b = \alpha(a) + k$. 
\end{lemma}

If $(\phi, {}_AX_A)$ is the correspondence associated to a topological graph $E$ and $I$ is an in-split of $E$,
then $I$ induces an in-split of correspondences in the sense of~\cref{defn:correspondence_insplit}. 
Moreover, the new correspondence associated to the in-split of correspondences may be identified with the graph correspondence of the in-split graph $E_I$.
It is in this sense that~\cref{defn:correspondence_insplit} generalises the topological notion of in-split of~\cref{defn:topological_graph_insplit}.

\begin{prop}\label{prop:graph_insplit_gives_correspondence_insplit}
Let $E$ be a topological graph and let $I = (\alpha,E^0_I,\psi)$ be an in-split of $E$.
Let $(\phi, X(E))$ and $(\phi_I, X(E_I))$ be the graph correspondences of $E$ and $E_I$, respectively.
Then there is an induced in-split $(\alpha^*, C_0(E_I^0), \psi^*)$ of $(\phi, X(E))$ satisfying 
\[
\alpha^*(f)= f \circ \alpha \quad \text{and} \quad \psi^*(f)x(e) = f(\psi(e))x(e)
\] 
for all $f \in A$, $x \in C_c(E^1)$, and $e \in E^1$. 

Moreover, the in-split correspondence $(\psi^* \ox \Id, X(E) \ox_{\alpha^*} C_0(E_I^0))$ is isomorphic to $(\phi_I, X(E_I))$.
\end{prop}

\begin{proof}
Let $A = C_0(E^0)$ and $A_I = C_0(E_I^0)$ be the coefficient algebras of $X(E)$ and $X(E_I)$, respectively.
Since $\alpha$ is a proper surjection there is a well-defined nondegenerate injective $*$-homomorphism $\alpha^* \colon A \to A_I$ given by $\alpha^*(f) = f \circ \alpha$ for all $f \in A$.
For each $g \in A_I$, define an endomorphism $\psi^*(g)$ on $C_c(E^1)$ by $\psi^*(g)x(e) \coloneqq g(\psi(e))x(e)$ for all $x \in C_c(E^1)$ and $e\in E^1$.
The computation 
\[
\|\psi^*(g)x\|^2 = \|(\psi^*(g)x \mid \psi^*(g)x)_A\|_{\infty} = \sup_{v \in E^0} \sum_{s(e)=v} |g(\psi(e))x(e)|^2 \le \|g\|_{\infty}^2 \|x\|^2,
\]
for all $x\in C_c(E^1)$ and $e\in E^1$
shows that the map $g \mapsto \psi^*(g)$ extends to a $*$-homomorphism $\psi^* \colon A_I \to \End_{A}(X(E))$ satisfying $\psi^*(g)^* = \psi^*(\bar{g})$.

Observe that $J_\phi = C_0(E^0_{\reg})$ and $J_{\psi} = C_0(E^0_{I,\psi-\reg})$, 
and since $\alpha$ restricts to a homeomorphism between $E^0_{\sing}$ and $E^0_{I,\psi-\sing}$, it follows that $\alpha$ maps $E_{I,\psi-\reg}$ onto $E^0_{\reg}$,
so $\alpha^*(J_\phi) \subseteq J_{\psi}$, and the induced map $\overline{\alpha} \colon C_0(E^0_{\sing}) \cong A/J_{\phi} \to A/J_{\psi} \cong C_0(E^0_{I,\psi-\sing})$ is a $*$-isomorphism.
Therefore, $(\alpha^*,A_I, \psi^*)$ is an in-split of the graph correspondence $(\phi, X(E))$. 

Next we verify that the $C^*$-correspondences $(\psi^* \ox \Id, X(E) \ox_{\alpha^*} C_0(E_I^0))$ and $(\phi_I, X(E_I))$ are isomorphic.
Define $\beta \colon C_c(E^1) \ox_{\alpha^*} C_c(E^0_I) \to C_c(E^1_I)$ by $\beta(x \ox f) (e,v) = x(e)f(v)$, for all $x\in C_c(E^1)$, $f\in C_c(E_I^0)$, and $(e,v)\in E_I^1$.
For $x,x\in C_c(E^1)$ and $f,f'\in C_c(E_I^0)$, the computation 
\begin{align*}
(\beta(x \ox f)\mid \beta(x' \ox f'))_{A_I}(v)
&= \sum_{s_I(e,v)=v} \ol{x(e)} x'(e) \ol{f(v)} f'(v) 
= \sum_{s(e) = \alpha(v)} \ol{x(e)}x'(e) \ol{f(v)}f'(v) \\
&= (x \mid x')_A(\alpha(v)) \ol{f(v)}f'(v)
= (f \mid \alpha^*((x \mid x')_A)f')_{A_{I}} (v)\\
&= (x \ox f \mid x' \ox f')_{A_I} (v),
\end{align*}
shows that $\|\beta(x \ox f)\| = \|x \ox f\|$. 
Consequently, $\beta$ extends to an isometric  linear map $\beta \colon X(E) \ox_{\alpha^*} A_I \to X(E_I)$.

If $x\in C_c(E^1)$ and $g,g'\in A_I$, then $\beta( (x\ox g)\cdot g') = \beta(x\ox g)\cdot g'$
and 
\begin{align*} 
\phi_I(g')\beta(x\ox g)(e,v) &=\! g'(\psi(e)) x(e) g(v) 
=\! \beta( (\psi^*\ox \Id)(g') x\ox g)(e,v),
\end{align*}
for all $(e,v)\in E_I^1$.
This shows that $(\Id, \beta)\colon (\psi^* \ox \Id, X(E) \ox_{\alpha^*} C_0(E_I^0)) \to (\phi_I, X(E_I))$ is a correspondence morphism.

It remains to verify that $\beta$ is surjective.
Fix $\eta \in C_c(E_I^1)$. 
Since $s_I$ is a local homeomorphism, we can cover $\supp(\eta)$ by finitely many open sets $\{U_i\}$ such that $s_I|_{U_i}$ is injective. Let $\{\rho_i\}$ be a partition of unity subordinate to the cover $\{U_i\}$.
Then $\rho_i\eta$ has support in $U_i$.

Define $\xi_i\in C_c(E^0_I)$ by $\xi_i(v)=\rho_i\eta(s^{-1}(v),v)$, and
use Urysohn's Lemma to find $\zeta_i\in C_c(E^1)$ such that $\zeta_i$ is identically 1 on the compact set
\[
\{e\in E^1:\,\mbox{there exists }v\in E^0_I\ \mbox{such that } (e,v)\in {\rm supp}(\rho_i\eta)\}.
\]
Then $\rho_i\eta=\beta(\zeta_i\ox \xi_i)$ by construction and so
\[
\eta=\sum_i\rho_i\eta=\sum_i\beta(\zeta_i\ox \xi_i)
\]
is in the image of $\beta$. As $\eta\in C_c(E^1_I)$ is arbitrary, $\beta$ is surjective.
\end{proof}

Every discrete directed graph is---in particular---a topological graph, so \cref{prop:graph_insplit_gives_correspondence_insplit} also applies to directed graphs. 
Since in-splits are examples of strong shift equivalences, ~\cref{thm:sse-equivariant-morita-equivalence} shows that the associated Cuntz--Pimsner algebras 
are gauge-equivariantly Morita equivalent.

\begin{prop} 
Let $(\phi, {}_AX_A)$ be a $C^*$-correspondence and let $(\alpha,B,\psi)$ be an in-split.
Then $(\phi, {}_AX_A)$ is 
strong shift equivalent to the in-split correspondence $(\psi\ox \Id, X\ox_\alpha B)$. Hence $\cO_X$ is gauge equivariantly Morita equivalent to $\cO_{X\ox_\alpha B}$.
\end{prop}

\begin{proof}
Consider the $C^*$-correspondences $R = (\psi, {}_B X_A)$ and $S = (\alpha, {}_A B_B)$ 
and observe that $S\ox R$ is isomorphic to $(\phi, {}_A X_A)$ via the map $b\ox x\mapsto \psi(b)x$ for all $x\in X_A$ and $b\in B$,
while $R\ox S$ is the in-split $(\psi\ox \Id, X\ox_\alpha B)$.
This is a strong shift equivalence.
\end{proof}

For in-splits more is true: 
they are gauge-equivariantly $*$-isomorphic (see \cref{thm:insplits_give_isomorpisms}), generalising \cite[Theorem 3.2]{Bates-Pask} (see also \cite[Theorem 3.2]{Eilers-Ruiz}).
First, we need a lemma.

\begin{lemma}\label{lem:beta_is_well_defined}
Let $X_A$ be a right Hilbert $A$-module and suppose that $\alpha \colon A \to B$ is an injective $*$-homomorphism. Then there is a well-defined injective linear map  $\beta\colon X \to X\ox_\alpha B$ satisfying $\beta(x\cdot a) = x\ox \alpha(a)$ for all $x\in X$ and $a\in A$. 

Moreover, suppose $\alpha$ is nondegenerate, $X_A$ is countably generated, and $A$ is $\sigma$-unital. Let $(x_i)_i$ be a countable frame for $X_A$ and let $(u_j)_{j}$ be an increasing approximate unit for $A$. With $a_j \coloneqq (u_j - u_{j-1})^{1/2}$ the sequence $(x_i \ox \alpha(a_j))$ is a frame for $X_A \ox_{\alpha}  B$.
\end{lemma}

\begin{proof}
For any $x\in X_A$ there is a unique $x'\in X_A$ such that $x = x' \cdot (x' \mid x')_A$, cf.~\cite[Proposition 2.31]{Raeburn-Williams},
so we may assume that any element in $X_A$ is of the form $x\cdot a$, for some $x\in X_A$ and $a\in A$.
Observe that for any $x_1, x_2\in X_A$ and $a_1,a_2\in A$, we have
\begin{equation} \label{eq:obs1}
(x_1\ox \alpha(a_1) \mid x_2\ox \alpha(a_2))_B = \alpha( (x_1\cdot a_1 \mid x_2\cdot a_2)_A),
\end{equation}
so 
\begin{align*}
\|x_1 \ox \alpha(a_1) - x_2 \ox \alpha(a_2)\|^2 
&= \|(x_1 \ox \alpha(a_1) \mid x_1 \ox \alpha(a_1)  )_B - (x_1 \ox \alpha(a_1) \mid x_2 \ox \alpha(a_2)  )_B\\
&\quad - (x_2 \ox \alpha(a_2) \mid x_1 \ox \alpha(a_1)  )_B + (x_2 \ox \alpha(a_2) \mid x_2 \ox \alpha(a_2)  )_B\|\\
&= \|\alpha( (x_1 \cdot a_1 \mid x_1 \cdot a_1)_A - (x_1 \cdot a_1 \mid x_2 \cdot a_2)_A \\  
&\quad - (x_2 \cdot a_2 \mid x_1 \cdot a_1)_A + (x_2 \cdot a_2 \mid x_2 \cdot a_2)_A)\|\\
&= \|x_1 \cdot a_1 - x_2 \cdot a_2\|^2.
\end{align*}
This computation shows that $\beta\colon X \to X\ox_\alpha B$ given by $\beta(x\cdot a) = x\ox \alpha(a)$ for all $x\in X$ and $a\in A$ is well-defined and isometric.

For the second statement, observe that $(a_i)_{i \in \N}$ is a frame for $A$ as a right $A$-module since
\[
\sum_{i=1}^j a_i \cdot (a_i \mid a)_A = \sum_{i=1}^j a_ia_i^*a = u_j a \to a
\]
as $j \to \infty$. The result now follows from \cref{prop:tensor_product_of_frames}.
\end{proof}

\begin{thm} \label{thm:insplits_give_isomorpisms}
Let $(\phi, {}_A X_A)$ be a countably generated correspondence over a $\sigma$-unital \mbox{$C^*$-algebra} $A$, let $(\alpha, B, \psi)$ be an in-split, 
and let $(\psi \ox \Id, X\ox_{\alpha} B)$ be the in-split correspondence. 
With the map $\beta$ as in \cref{lem:beta_is_well_defined}, the pair $(\alpha,\beta)\colon (\phi, X) \to (\psi \ox \Id, X\ox_\alpha B)$ is a covariant correspondence morphism.
The induced $*$-homomorphism $\alpha\times \beta\colon \cO_X \to \cO_{X\ox_\alpha B}$ is a gauge-equivariant $*$-isomorphism.
\end{thm}

\begin{proof}
We first verify that $(\alpha, \beta)\colon (\phi, X) \to (\psi \ox \Id, X\ox_\alpha B)$ is a correspondence morphism.
For the right action, we see for $x\in X_A$ and $a,a'\in A$ that
\[
\beta(x\cdot a)\cdot \alpha(a') = x\ox \alpha(a a') = \beta( (x\cdot a) \cdot a'),
\]
and for the left action, we apply $\psi\circ \alpha = \phi$ to observe that
\[
(\psi\ox \Id)(\alpha(a')) \beta(x\cdot a) = \phi(a') x \ox \alpha(a) = \beta(\phi(a') x\cdot a),
\]
for all $x\in X_A$ and $a,a'\in A$.
Together with~\labelcref{eq:obs1} this shows that $(\alpha,\beta)$ is a correspondence morphism.  

For covariance of $(\alpha, \beta)$ let  $(x_i \ox \alpha(a_j))$ be the frame for $X_A \ox_{\alpha}  B$ as defined in \cref{lem:beta_is_well_defined}.
Then for $T \in \End_A^0(X)$, 
\begin{equation}\label{eq:beta1_is_tensor_id}
\beta^{(1)}(T) 
= \sum \Theta_{\beta( T x_i \cdot a_{j}), \beta(x_i \cdot a_{j})} 
= (T \ox \Id_B) \sum \Theta_{x_i \ox \alpha(a_j), x_i \ox \alpha(a_j)}
= T \ox \Id_B.
\end{equation}
Let $a \in J_{\phi}$. Then setting $T = \phi_X(a)$ we have 
\[
\beta^{(1)} \circ \phi_X(a) = \phi_X(a) \ox \Id_B =(\psi\circ\alpha(a) )\ox \Id_B= (\psi \ox \Id_B) \circ \alpha (a)
\] 
so $(\alpha,\beta)$ is covariant. 
Since $\alpha$ is injective, we know from~\cref{lem:induced-gauge-invariant-homomorphism} that $\alpha\times \beta$ is injective and gauge-equivariant. 

For surjectivity we first claim that $\iota_B(B)$ lies in the image of $\alpha \times \beta$. 
Fix $b \in B$ and write $b = \alpha(a) + k$ for some $a \in A$ and $k \in J_{\psi} = J_{\psi\ox \Id_B}$ using \cref{lem:b_decomposition}.
Since $\psi(k)$ is compact, we get $\beta^{(1)} (\psi(k)) = \psi(k) \ox \Id_B$. 
It then follows from covariance of $(\iota_B,\iota_{X \ox_{\alpha} B})$ that 
\[
\iota_B(k) = \iota^{(1)}_{X \ox_{\alpha} B} \circ (\psi \ox \Id_B)(k) = \iota^{(1)}_{X \ox_{\alpha} B} \circ \beta^{(1)} (\psi(k)) = (\alpha \times \beta) \circ \iota_X^{(1)} (\psi(k))
\]
also lies in the image of $\alpha \times \beta$. Consequently,
\[
\iota_B(b) = \iota_B(\alpha(a)) + \iota_B(k) = (\alpha \times \beta)(\iota_A(a)) + \iota_B(k) \in (\alpha \times \beta) (\cO_X).
\] 
Finally, observe that if $x\cdot a \ox b \in X \ox_A B$, then
\[
\iota_{X \ox_\alpha B}(x\cdot a \ox b) = (\alpha \times \beta)   (\iota_X(x \cdot a)) \iota_B(b)
\]
which is in the image of $\alpha\times \beta$.
This shows that $\alpha\times \beta$ is surjective, and we conclude that it is a gauge-equivariant $*$-isomorphism.
\end{proof}

\begin{example} \label{ex:maximal_insplit}
Let $(\phi, {}_AX_A)$ be a regular $C^*$-correspondence and let $(\alpha,B,\psi)$ be an in-split.
Since both $\phi$ and $\alpha$ are injective, we may identify $B$ with a subalgebra of $\End_A^0(X)$ that contains $\phi(A)$.
Conversely, any $C^*$-algebra $B$ satisfying $\phi(A) \subset B \subset \End^0_A(X)$ determines an in-split $(\psi, B, \phi)$ where $\psi\colon B \to \End_A(X)$ is the inclusion.
Therefore, there is a gauge-equivariant $*$-isomorphism $\cO_X \cong \cO_{X\ox_{\phi} B}$.
In particular---as noted in \cite[Example 6.4]{Menevse}---there is a gauge-equivariant $*$-isomorphism $\cO_X \cong \cO_{X\ox_{\phi} \End_A^0(X)}$.
\end{example}

Consider a regular correspondence $(\phi,{}_A X_A)$ and let $i \colon \End_A^0(X) \to \End_A(X)$ denote the inclusion. 
Then $(i \ox \Id_{\End_A^0(X)}, X \ox \End_A^0(X))$ may be thought of as a ``maximal'' in-split of $(\phi,{}_A X_A)$ in analogy to the dual graph in the setting of topological graphs. 
This analogy is further justified by the following noncommutative version of \cref{lem:top_graph_diamond_lemma}. 
\begin{lemma}
Let $(\phi,{}_A X_A)$ be a regular nondegenerate $C^*$-correspondence with an in-split $I = (\alpha,B,\psi)$. Let $\alpha_1 \colon \End_A^0(X) \to \End_A^0(X \ox_{\alpha} B)$ be the map defined by $\alpha_1(T) = T \ox \Id_B$. Then $I' = (\psi, \End_A^0(X), \alpha_1)$ is an in-split of $(\psi \ox \Id_B, X \ox_{\alpha} B)$ and 
\[
\big(\alpha_1 \ox \Id_{\End_A^0(X)}, (X \ox_{\alpha} B)\ox_{\psi} \End_A^0(X)\big)
\cong
\big(i \ox \Id_{\End_A^0(X)}, X \ox_{\phi} \End_A^0(X)\big)
\]
as $\End_A^0(X)$--$\End_A^0(X)$-correspondences.
\end{lemma}

\begin{example}
\label{eg:circles-CP}
Let $E = (E^0,E^1,r,s)$ be the topological graph of \cref{eg:n-m-circles}, with $r (z)= z^m$ and $s(z) = z^n$ and let  $E_I = (E_I^0,E_I^1,r_I,s_I)$ be the in-split graph of \cref{eg:n-m-circles}. 
In particular, $E_I^0 = \bT$, $E_I^1 = \bigsqcup_{k=0}^{\gcd(n,b)-1} \bT$, $r_I(k,z) = z^{m/\gcd(n,b)}$ and $s_I(k,z) = \lambda^k z^{n/\gcd(n,b)}$ for some fixed $|b|$-th root of unity $\lambda$. 
Then $\cO_{X(E)}$ and $\cO_{X(E_I)}$ are gauge equivariantly $*$-isomorphic.

Consider $E$ when $m = n = 2$, and define a directed graph $F = (F^0,F^1,r_F,s_F)$ with vertices $F^0= \bT$, edges $F^1 = \{0,1\} \times \bT$, and $r_F(z) = s_F(z) = z$. 
It is shown in \cite[\S 5]{Frausino-Ng-Sims} that the graph correspondence $X(E)$ is isomorphic to the graph correspondence $X(F)$, while $E$ and $F$ are not isomorphic as graphs. 

Let $a = 1$ and $b = 2$, so $\gcd(n,b) = 2$. 
The in-split $E_I$ has vertices $E_I^0 = \bT$ and edges $E_I^1 = \{0,1\} \times \bT$ with $r_I(k,z) = (-1)^k z^2$ and $s_I(k,z) = z^2$. 
Here we have used the second description of $E_I$ from \cref{eg:n-m-circles}. 
Since the edges, vertices, and source maps are the same for both $E_I$ and $F$, it follows that $X(E_I)$ and $X(F)$ are isomorphic as right $C(\bT)$-modules. 
On the other hand, since the range maps on $E_I$ and $F$ are different, the left action of $C(\bT)$ differs between the two modules. 
In particular, $X(E_I)$ is not isomorphic to $X(F) \cong X(E)$ as $C^*$-correspondences. 
We suspect that $X(E_I)$ is typically not isomorphic to $X(E)$ in general.
\end{example}

\subsection{In-splits and diagonal-preserving isomorphism}
\label{subsec:diag-iso}

The work of Eilers and Ruiz \cite[Theorem 3.2]{Eilers-Ruiz} shows that unital graph algebras of in-splits (out-splits in their terminology) 
are gauge-equivariantly $*$-isomorphic in a way that also preserves the canonical diagonal subalgebras.
In our general setting of Cuntz--Pimsner algebras, there is no obvious notion of canonical diagonal subalgebras. 
However, specialising to the setting of topological graphs, we can define such a diagonal. We prove in \cref{prop:diagonal_preserving_iso} that in-splits of correspondences over topological graphs gives a diagonal-preserving and gauge-equivariant $*$-isomorphism of the corresponding Cuntz--Pimsner algebras.

\begin{lemma}\label{lem:insplit_fock}
Let $(\phi, X_A)$ be a nondegenerate $C^*$-correspondence over $A$ and let $(\alpha,B,\psi)$ be an in-split. 
Then $(X \ox_{\alpha} B)^{\ox k} \cong X^{\ox k} \ox_{\alpha} B$ as right $B$-modules via the isomorphism
\[
x_1 \ox b_1 \ox \cdots \ox x_k \ox b_k \mapsto x_1 \ox \psi(b_1) x_2 \ox \cdots \psi(b_{k-1}) x_k \ox b_k.
\]
In particular, $\Fock(X \ox_{\alpha} B) \cong \Fock(X) \ox_{\alpha} B$.
\end{lemma}
\begin{proof}
Since $\phi$ and $\alpha$ are nondegenerate, so is $\psi$. 
Hence, $X \ox_{\alpha} B \ox_\psi X \cong X \ox_{\psi \circ \alpha} X = X^{\ox 2}$ via the map $x_1 \ox b_1 \ox x_2 \mapsto x_1 \ox \psi(b_1)x_2$. 
The result now follows inductively.
\end{proof}

We now restrict to topological graphs and show that there is a notion of diagonal subalgebra.

\begin{lemma}
Let $E$ be a topological graph and let 
\[
\cC^1_E = \{ x \in C_c(E^1) \mid x \ge 0 \text{ and } s|_{\supp(x)} \text{ is injective}\}.
\]
Then
$
\cD^1_E \coloneqq \ol{\linspan}\{ \Theta_{x,x} \mid x \in \cC^1_E \}
$
is a commutative subalgebra of $\End_{C_0(E^0)}^0(X(E))$ which is isomorphic to $C_0(E^1)$. 
\end{lemma}
\begin{proof}
Since $E^1$ is paracompact and $s$ is locally injective we may  choose a locally finite open cover $\{U_i\}$ of $E^1$ such that $s|_{U_i}$ is injective. Let $\{\rho_i\}$ be a partition of unity subordinate to $\{U_i\}$. Fix a positive function $x \in C_0(E^1)$ and define
\[
\Psi(x) 
\coloneqq \sum_{i=1}^{\infty} \Theta_{(\rho_ix)^{1/2},(\rho_ix)^{1/2}} \in \cD^1_E.
\]
The sum converges as for $z \in C_c(E^1)$ and $e \in E^1$,
\begin{equation}\label{eq:mult_operator}
\Psi(x)z(e) = \sum_{i=1}^\infty (\rho_ix)^{1/2}(e) \sum_{s(f) = s(e)} (\rho_ix)^{1/2}(f)z(f) = \sum_{i=1}^{\infty} \rho_i(e) x(e) z(e) = x(e)z(e).
\end{equation}
In particular, $\Psi(x)$ acts as a multiplication operator. 
Moreover, \eqref{eq:mult_operator} implies that $\Psi(x)$ is independent of the choice of open cover and partition of unity. Since $\Psi(x+y) = \Psi(x)+\Psi(y)$, and $\Psi(xy) = \Psi(x)\Psi(y)$ for positive $x,y \in C_0(E^1)$ and positive elements span $C_0(E^1)$ we can linearly extend the formula $\Psi(x)$ to all $x \in C_0(E^1)$ to obtain a $*$-homomorphism $\Psi \colon x \mapsto \Psi(x)$ from $C_0(E^1)$ to $\cD_E^1$. 
Since $|\supp(x) \cap s^{-1}(v)| \le 1$ for all $x \in \cC^1_E$ and $v \in E^0$, it follows that
\begin{align*}
\|\Psi(x)\|^2 = \sup_{\|z\|=1} \sup_{v \in E^0} \sum_{s(e) = v} |\Psi(x)z(e)|^2 = \sup_{\|z\|=1} \sup_{v \in E^0} \sum_{s(e) = v} |x(e)z(e)|^2.
\end{align*}
Since $\sup_{\|z\|=1}  \sum_{s(e) = v} |x(e)z(e)|^2$ is the square of the operator norm of the multiplication operator $\Psi(x)$ restricted to $\ell^2(s^{-1}(v))$ it follows that $\|\Psi(x)\|^2 = \|x\|_{\infty}^2$ and so $\Psi$ is isometric. For surjectivity observe that for each $x \in \cC_E^1$, 
\[
\Theta_{x,x} z(e) = x^2(e)z(e) = \Psi(x^2)z(e)
\]
so $\Theta_{x,x} = \Psi(x^2)$. Since the $\Theta_{x,x}$ densely span $\cD_E^1$ surjectivity of $\Psi$ follows.
\end{proof}

\begin{defn}
Let $E$ be a topological graph. We call $\cD^1_E \subseteq \End_{C_0(E^0)}^0(X(E))$ the \emph{diagonal} of $\End_{C_0(E^0)}^0(X(E))$.  For $k \ge 1$ define
\begin{align*}
\cC^k_E &\coloneqq \{ x \in C_c(E^k) \mid x \ge 0 \text{ and } s|_{\supp(x)} \text{ is injective}\} \quad \text{ and }\\
\cD^k_E &\coloneqq \ol{\linspan}\{\Theta_{x,x} \mid x \in \cC^k_E\} \cong C_0(E^k).
\end{align*}
Let $\cD^0_E = C_0(E^0)$. Define the \emph{diagonal} of $\cO_{X(E)}$ to be the $C^*$-subalgebra
\begin{align*}
\cD_{E}
&\coloneqq \sum_{k = 0}^{\infty} \iota^{(k)}_{X(E)} (\cD_E^k)
= \ol{\linspan}\{\iota^{(k)}_{X(E)} (\Theta_{x,x}) \mid x \in \cC^k_E, k \ge 0\}\\
&= \ol{\linspan}\{\iota^{(k)}_{X(E)} (x) \iota^{(k)}_{X(E)}(x)^* \mid x \in \Cc^k_E, k \ge 0\},		
\end{align*}
where the terms of the sum are not necessarily disjoint. 
\end{defn}

\begin{rmk}
For each $k \ge 0$, $\End_{C_0(E^0)}^0(X(E^k))$ is isomorphic to the groupoid $C^*$-algebra of the amenable \'etale groupoid $\cR_k \coloneqq \{ (x,y) \in E^k \times E^k \mid s(x) = s(y)\}$. 
The isomorphism $\Phi \colon \End_{C_0(E^0)}^0(X(E^k)) \to C^*(\cR_k)$ satisfies $\Phi(\Theta_{x,y})(e,f) = x(e)\ol{y(f)}$ for $x,y \in X(E)$, with the inverse satisfying
\[
\Phi^{-1}(\xi)x(e) = \sum_{s(f)=s(e)} \xi(e,f) x(f)
\]
for $\xi \in C_c(\cR_k)$ and $x \in C_c(E^1)$. 
Details can be found in \cite[Proposition 3.2.14]{MundeyPhD}. 
The map $\Phi$ takes $\cD_k$ onto the canonical diagonal subalgebra of $C^*(\Rr_k)$ consisting of $C_0$-functions on the unit space. 
Moreover, $\cD_E$ is the canonical diagonal of the Deaconu--Renault groupoid associated to $E$ \cite[Proposition 3.3.16]{MundeyPhD}.
\end{rmk}

Recall from \cref{prop:graph_insplit_gives_correspondence_insplit} that if $(\alpha, E_I^0, \psi)$ is an in-split of a topological graph $E$, then $(\alpha^*, C_0(E_I)^0, \psi^*)$ is an in-split of the graph correspondence $(\phi,X(E))$.
Since \cref{prop:graph_insplit_gives_correspondence_insplit} implies $X(E_I) \cong X(E) \ox_{\alpha^*} C_0(E_I^0)$,
we may consider the $*$-isomorphism $\alpha^* \times \beta$ of \cref{thm:insplits_give_isomorpisms} as a map $\alpha^* \times \beta \colon \cO_{X(E)} \to \cO_{X(E_I)}$. 
We will show that $\alpha^* \times \beta$ also preserves diagonals in the sense that $(\alpha^* \times \beta) (\cD_E) = \cD_{E_I}$.

To this end, let $B = C_0(E^0_I)$. It follows from \cref{lem:insplit_fock} that there are $*$-isomorphisms
\[
X(E)^{\ox k} \ox_{\alpha^*} B \cong  (X(E) \ox_{\alpha^*} B)^{\ox k} \cong X(E_I)^{\ox k} \cong X(E_I^k)
\]
for all $k \ge 1$. 
Using \cref{lem:insplit_paths} to identify $E_I^k$ with $E^k \times_{s,\alpha} E_I^0$ we define
\[
(x_1 \ox \cdots \ox x_k \ox b)(e_1,\ldots,e_k, v) \coloneqq x_1(e_1)\cdots x_k(e_k) b(v),
\]
for all $x_1,\dots,x_k\in X(E)$, $b\in B$, and $(e_1,\dots,e_k, v)\in E^k\times_{s,\alpha} E_I^0$.

\begin{lemma}\label{lem:diagonal_compatibility} Let $x \in C_c(E^k)$ and $b \in C_0(E^0_I)$. Then
$x \ox b \in \Cc^k_{E_I}$ if and only if $x \in \Cc^k_E$.
\end{lemma}
\begin{proof}
Fix $x \in \Cc_E^k$ and $b \in C_0(E^0_I)$. If $(x \ox b)(e,v) = x(e)b(v)$ and $(x\ox b)(e',v) = x(e')b(v)$ are nonzero for $e,e' \in E^k$ with $s(e) = s(e') = \alpha(v)$, then $x(e)$ and $x(e')$ are nonzero, so by assumption $e = e'$ and hence $(x \ox b) \in \Cc_{E_I}^k$. Conversely, suppose $(x \ox b) \in \Cc_{E_I}^k$ and $x(e)$ and $x(e')$ are nonzero for some $e,e' \in E^k$ with $s(e) = s(e')$. Then $(x \ox b)(e,v)$ and $(x \ox b)(e',v)$ are both nonzero as soon as one is nonzero. Hence, $e = e'$ and so $x \in \Cc_E^k$. 
\end{proof}
\begin{prop}\label{prop:diagonal_preserving_iso}
Let $E$ be a topological graph and let $I = (\alpha,E_0^I,\psi)$ be an in-split of $E$.
Then the Cuntz--Pimsner algebras $\cO_{X(E)}$ and $\cO_{X(E_I)}$ are gauge-equivariantly $*$-isomorphic
in a way that also preserves the diagonal subalgebras.
\end{prop}

\begin{proof}
Since $\alpha^* \times \beta$ is injective, it is enough to show that $(\alpha^* \times \beta)(\cD_E) = \cD_{E_I}$. 	
Let $a_j = (u_j - u_{j-1})^{1/2}$ be as in the statement of \cref{lem:beta_is_well_defined}, and recall that ($\alpha^*(a_j))_j$ is a frame for $B$ as a right Hilbert $B$-module. Let $\beta^{(k)}$ denote the map $(\beta^k)^{(1)}  \colon \End^0_{A}(X(E)^{\ox k}) \to \End^0_B(X(E_I)^{\ox k})$. 
Given $x \in \cC_{E}^k$ we may apply \eqref{eq:beta1_is_tensor_id} to $\Theta_{x,x} \in \End_A^0(X^{\ox k})$ to see that
\[
\beta^{(k)} (\Theta_{x,x}) = \Theta_{x,x} \ox \Id_B = \sum_{i=1}^{\infty} \Theta_{x \ox \alpha^*(a_i), x \ox \alpha^*(a_i)}.
\] 
It follows from \cref{lem:diagonal_compatibility} that $x \ox \alpha^*(a_i) \in \cC_{E_I}^k$ so $\beta^{(k)} (\Theta_{x,x}) \in \cD_{E_I}^k$. Consequently,
\begin{align*}
(\alpha^* \times \beta) \circ \iota_{X(E)}^{(k)} (\Theta_{x,x}) = \iota_{X(E_I)}^{(k)} \circ \beta^{(k)} (\Theta_{x,x}) \in \cD_{E_I}
\end{align*}
and so $(\alpha^* \times \beta)(\cD_E) \subseteq \cD_{E_I}$. 

For surjectivity, first observe that since $X(E_I)^{\ox k} \cong X(E)^{\ox k} \ox_{\alpha^*} B$ is densely spanned by the set $\{x \ox b \mid x \in X(E)^{\ox k},\, b \in B\}$, 
and \cref{lem:diagonal_compatibility} states that $x \ox b \in \Cc^k_{E_I}$ if and only if $x \in \Cc^k_{E}$, so 
\[
\cD_{E_I}^k = \ol{\linspan}\{\Theta_{x \ox b, x \ox b} \mid x \in \Cc^k_E,\, b \in C_0(E^0_I)\}.
\]

Observe that for $x \ox b \in \Cc_{E_I}^k$,
\begin{align*}
\iota_{X \ox_\alpha B}^{(k)} (\Theta_{x \ox b,x \ox b}) = \iota_{X \ox_{\alpha} B}^{k}(x \ox b)\iota_{X \ox_{\alpha} B}^{k}(x \ox b)^*
&=( \alpha \times \beta)(\iota_{X}^k(x)) \iota_{B} (bb^*) ( \alpha \times \beta)(\iota_{X}^k(x))^*,
\end{align*}
so it suffices to show that $\iota_{B}(b) \in (\alpha \times \beta)(\cD_E)$ for each $b \in B$.

Fix $b \in B$ and use \cref{lem:b_decomposition} to write $b = \alpha^*(a) + j$ for some $a \in A$ and $j \in J_{\psi}$. 
We have $\iota_B(\alpha^*(a)) = (\alpha^* \times \beta)(\iota_A(a)) \in (\alpha^* \times \beta)(\cD_E)$. 
On the other hand, when $j \ge 0$,
\[
\psi^*(j) = 	\psi^*(j)^{1/2} \sum_{i=1}^{\infty}\Theta_{x_i,x_i} \psi^*(j)^{1/2} = \sum_{i = 1}^{\infty} \Theta_{\psi^*(j)^{1/2} x_i, \psi^*(j)^{1/2} x_i}.
\]
Since $(\psi^*(j)^{1/2} x_i) (e) = j(\psi(e))^{1/2}x_i(e)$ it follows that $s$ restricted to the support of $ \psi^*(j)^{1/2} x_i$ is injective. Hence, $\psi(j) \in \cD_E^1$ and by linearity this is also true for general $j \in J_{\psi}$. 
Covariance of $(\iota_B,\iota_{X \ox_\alpha B})$ and \eqref{eq:beta1_is_tensor_id} imply that
\begin{align*}
    \iota_B (j) 
     =\sum_i \iota_{X \ox_{\alpha} B}^{(1)}  (\Theta_{\psi^*(j)^{1/2} x_i, \psi^*(j)^{1/2} x_i} \ox \Id_B)
    &= \sum_i \iota_{X \ox_{\alpha} B}^{(1)} \circ \beta^{(1)} (\Theta_{\psi^*(j)^{1/2} x_i, \psi^*(j)^{1/2} x_i} )\\
    &= \sum_i (\alpha^* \times \beta) \circ \iota_X^{(1)}(\Theta_{\psi^*(j)^{1/2} x_i, \psi^*(j)^{1/2} x_i})
\end{align*}
belongs to $(\alpha^* \times \beta)(\cD_{E})$. Consequently, $\iota_B(b) \in (\alpha^* \times \beta)(\cD_E)$ and so $(\alpha^* \times \beta)(\cD_E) = \cD_{E_I}$. 
\end{proof}

\begin{example}
The $*$-isomorphism between $\cO_{X(E)}$ and $\cO_{X(E_I)}$ of \cref{eg:circles-CP} is also diagonal preserving.
\end{example}

\begin{rmk}
\cref{thm:sse-equivariant-morita-equivalence} and \cref{prop:diagonal_preserving_iso} imply that a diagonal-preserving, gauge-equivariant $*$-isomorphism 
between the Cuntz--Pimsner algebras of topological graphs is \emph{not} sufficient to recover the original $C^*$-correspondence up to isomorphism. 
An analogous result for Cuntz--Pimsner algebras of graph correspondences states that diagonal-preserving, gauge-equivariant isomorphisms are \emph{not} sufficient to recover the graph up to conjugacy.

The final section of \cite{Brix-Carlsen} exhibits an example of a pair of finite and strongly connected graphs that are not conjugate
but whose graph $C^*$-algebras admit a $*$-isomorphism that is both gauge-equivariant and diagonal-preserving.
The main result of \cite{ABCE} uses groupoid techniques to recover a topological graph up to conjugacy using $*$-isomorphisms that intertwine a whole family of gauge actions. 
For general Cuntz--Pimsner algebras there is no obvious such family of gauge actions.

A recent preprint \cite{Frausino-Ng-Sims} explains how to recover the graph correspondence of a compact topological graph 
from its Toeplitz algebra, its gauge action, and the commutative algebra of functions on the vertex space.
\end{rmk}

\section{Out-splits}
\label{sec:out-damn-split}

In this section, we consider the dual notion of an out-split.
The non-commutative version applied to Cuntz--Pimsner algebras is not as fruitful as non-commutative in-splits. The inputs are more restrictive and the outputs less exciting, but we include this section for completeness.

For a graph, we will see that an out-split corresponds to a factorisation of the source map.
We use the notation of Bates and Pask~\cite{Bates-Pask} as well as Eilers and Ruiz~\cite{Eilers-Ruiz},
but we warn the reader that our graph conventions follow Raeburn's monograph~\cite{Raeburn} and so are \emph{opposite} to the convention used in those papers.

\subsection{Out-splits for directed graphs}
Let $E = (E^1, E^0, r, s)$ be a countable discrete directed graph.
We recall the notion of an out-split from~\cite[Section 3]{Bates-Pask}.
Fix a regular $w\in E^0$ (i.e. $0 < |s^{-1}(w)| < \infty$),
and let $\{\Pp^i\}_{i=1}^n$ be a partition of $s^{-1}(w)$ into finitely many (possibly empty) sets. 

The \emph{out-split graph of $E$ associated to $\Pp$} is the graph $E_r(\Pp)$ given as
\begin{align*}
E_s(\Pp)^0 &= \{ v_1 : v\in E^0 \} \cup \{w_1,\ldots,w_n \}\\
E_s(\Pp)^1 &= \{ e_1 : e\in E^1, r(e) \neq w \} \cup \{e_1,\ldots,e_n : e \in E^1, r(e) = w\},  \\
r_\Pp(e_j) &= r(e)_j,  \\
s_\Pp(e_j) &= 
\begin{cases}
s(e)_1 & \textrm{if } s(e) \neq w, \\
w_i & \textrm{if } s(e) = w \textrm{ and } e\in \Pp^i,
\end{cases}
\end{align*}
for all $e_j\in E_s^1(\Pp)$.

\begin{example}
Consider the graphs
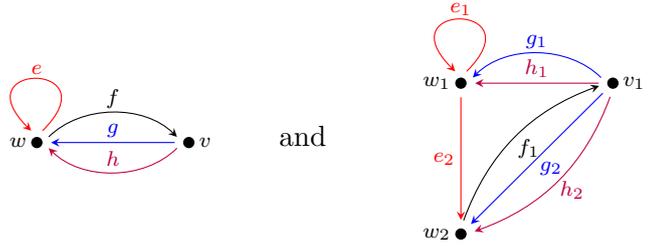
\[
\begin{tikzpicture}
[baseline=-0.25ex,
vertex/.style={
circle,
fill=black,
inner sep=1.5pt
},
edges/.style={
-stealth,
shorten >= 3pt,
shorten <= 3pt
},
scale =1]

\node[vertex] (a) at (0,0) {};%
\node[vertex] (b) at (2,0) {};%

\node[anchor= east] at (a) {\scriptsize{$w$}};
\node[anchor= west] at (b) {\scriptsize{$v$}};

\draw[edges,blue] (b.west) -- node[anchor=south, inner sep = 2pt]{\scriptsize{$g$}} (a.east);

\draw[edges,purple] (b.west) to [out = 225, in = -45] node[anchor=south, inner sep = 2pt]{\scriptsize{$h$}} (a.east);

\draw[edges] (a.east) to [out = 45, in = -225] node[anchor=south, inner sep = 2pt]{\scriptsize{$f$}} (b.west);

\draw[edges,red] (a.north) to [out = 45, in = -225, min distance = 40pt, looseness = 10] node[anchor=south, inner sep = 2pt]{\scriptsize{$e$}} (a.north);
\end{tikzpicture}
\qquad
\text{and}
\qquad 
\begin{tikzpicture}
[baseline=-5ex,
vertex/.style={
circle,
fill=black,
inner sep=1.5pt
},
edges/.style={
-stealth,
shorten >= 3pt,
shorten <= 3pt
},
scale =1]

\node[vertex] (a) at (0,0) {};%
\node[vertex] (b) at (2,0) {};%
\node[vertex] (c) at (0,-2) {};%

\node[anchor= east] at (a) {\scriptsize{$w_1$}};
\node[anchor= west] at (b) {\scriptsize{$v_1$}};
\node[anchor= east] at (c) {\scriptsize{$w_2$}};

\draw[edges,blue] (b.south west) -- node[anchor=north west, inner sep = 1pt]{\scriptsize{$g_2$}} (c.north east);

\draw[edges,purple] (b.west) to node[anchor=south, inner sep = 2pt]{\scriptsize{$h_1$}} (a.east);
\draw[edges,blue] (b.west) to [out = -225,in = 45] node[anchor=south, inner sep = 2pt]{\scriptsize{$g_1$}} (a.east);

\draw[edges,purple] (b.south) to [out = 250, in = 20] node[anchor=north west, inner sep = 0.5pt]{\scriptsize{$h_2$}} (c.east);

\draw[edges] (c.north) to [out = 70, in = 200] node[anchor=north west, inner sep = 0.5pt]{\scriptsize{$f_1$}} (b.west);

\draw[edges,red] (a.north) to [out = 45, in = -225, min distance = 40pt, looseness = 10] node[anchor=south, inner sep = 2pt]{\scriptsize{$e_1$}} (a.north);
\draw[edges,red] (a.south) to node[anchor=east, inner sep = 2pt]{\scriptsize{$e_2$}} (c.north);
\end{tikzpicture}.
\]
The incoming edges to $w$ are coloured for clarity.
Then $s^{-1}(w) = \{ e, f\}$ and we consider the partition $\Pp_1 = \{e\}$ and $\Pp_2 = \{f\}$.
The out-split graph---with respect to this partition---is the right-most graph above.

Note that the loop $e$ is both an incoming and an outgoing edge.
The adjacency matrices of the graphs are 
\[
\textsf{A} = 
\begin{pmatrix}
1 & 1 \\
2 & 0
\end{pmatrix} \qquad \textrm{and} \qquad
\textsf{C} = 
\begin{pmatrix}
1 & 1 & 0 \\
0 & 0 & 1 \\
2 & 2 & 0
\end{pmatrix}
\]
and the rectangular matrices 
\[
\textsf{R} = 
\begin{pmatrix}
1 & 0 \\
0 & 1 \\
2 & 0
\end{pmatrix} \quad \textrm{and} \quad
\textsf{S} = 
\begin{pmatrix}
1 & 1 & 0 \\
0 & 0 & 1
\end{pmatrix}
\]
satisfy $\mathsf{C = SR}$ and $\mathsf{RS = A}$.
Therefore, $\textsf{A}$ and $\textsf{C}$ are (elementary) strong shift equivalent. 
Any out-split induces a strong shift equivalence, cf \cite[Chapter 7]{Lind-Marcus}. 
\end{example}

The out-split at $w$ can be summarised as two pieces of information:
there is a finite-to-one surjection $ \alpha \colon E_s^0(\Pp) \to E^0$ given by 
$
\alpha(v_j) = v,
$
for all $v_j\in E_s^0(\Pp)$,
and a surjection $\psi\colon E^1 \to E_s^0$ given by
\[
\psi(e) =
\begin{cases}
s(e)_1 & \textrm{if } s(e) \neq w, \\
w_i  & \textrm{if } s(e) = w, e\in \Pp^i,
\end{cases}
\]
for all $e\in E^1$.
Observe that $s = \alpha \circ \psi$, so we interpret an out-split as a factorisation of the source map
(in contrast to an in-split which we saw was a factorisation of the range map).

We may now form the graph  $(E_s^0(\Pp), E_s^0(\Pp)\times_{\alpha, r} E^1, r, s)$
where the edge set is the fibred product 
\[
E_s^0(\Pp)\times_{\alpha, r} E^1 = \{ (v_j, e)\in E_s^0(\Pp)\times E^1 : v = r(e) \}
\]
and $r(v_j,e) = v_j$ and $s(v_j,e) = \psi(e)$ for all $(v^j,e)\in E_s^0(\Pp)\times_{\alpha,r}E^1$.
This is graph isomorphic to the out-split graph $E_s(\Pp)$ via the map $e_j \mapsto (v_j, e)$ for all $e^j\in E_s^1(\Pp)$.

We give a definition of out-splits for regular topological graphs, which includes regular directed graphs. 

\begin{defn}
\label{defn:topological_graph_outsplit}
An \emph{out-split} (or \emph{source-split}) of a topological graph $E = (E^0,E^1,r,s)$ is a triple $\bO = (\alpha,Y,\psi)$ consisting of
\begin{enumerate}
\item a locally compact Hausdorff space $Y$,
\item a proper surjective local homeomorphism $\alpha \colon Y \to E^0$, and
\item a proper surjective local homeomorphism $\psi \colon E^1 \to Y$, 
\end{enumerate}
such that  $\alpha \circ \psi = s$.
\end{defn}

\begin{rmk}
The continuity assumptions of an out-split $\bO = (\alpha, E^0_\bO,\psi)$ are automatic for regular directed graphs.
\end{rmk}

We associate a new topological graph to an out-split.

\begin{lemma} 
Let $E = (E^0,E^1,r,s)$ be a regular topological graph and let $\bO = (\alpha, Y,\psi)$ be an out-split of $E$.
Then $E_\bO = (E_\bO^0, E_\bO^1,r_\bO,s_\bO)$ is a regular topological graph, where 
\begin{enumerate}
\item $E^0_\bO \coloneqq Y$;
\item $E^1_\bO \coloneqq E^0_\bO \times_{\alpha,r} E^1 = \{ (v,e) \in E^0_\bO \times E^1 \mid  \alpha(v) = r(e)\}$ equipped with the subspace topology of the product $E^0_\bO \times E^1$; and
\item $r_\bO(v,e) = v$ and $s_\bO(v,e) = \psi(e)$, for all $e\in E^1$ and $v\in E_\bO^0$.
\end{enumerate}
\end{lemma}
\begin{proof}
We will be brief as the proof is similar to the in-split case.
The edge space $E^1_\bO$ is a closed subspace of a locally compact Hausdorff space, and so is locally compact and Hausdorff. Also $s_\bO$ is a local homeomorphism since $\psi$ and $\alpha$ are.

The map $r_\bO$ is clearly continuous and is surjective since $r$ is surjective. The range $r_\bO$ is proper, and to see this we let $K\subset E^0_\bO$ be compact. Then 
\[
r_\bO^{-1}(K)=K\times_{\alpha,r}r^{-1}(\alpha(K))
\]
is compact. So $E_\bO$ is a regular topological graph.
\end{proof}

\begin{defn}
We call $E_\bO = (E_\bO^0,E_\bO^1, r_\bO, s_\bO)$ the \emph{out-split graph of $E$ via $\bO$}.
\end{defn}

\subsection{Noncommutative out-splits}

In-splits for topological graphs correspond to factorisations of the range map. In the noncommutative setting this translates to a factorisation of the left action on the associated graph correspondence. On the other hand, out-splits for topological graphs correspond to a factorisation of the source map, which defines the right-module structure of the graph correspondence. This makes the noncommutative analogy for out-splits more difficult to pin down than in the case of in-splits.

\begin{defn}\label{def:nc-outsplit}
An \emph{out-split} of a regular $C^*$-correspondence $(\phi_X, {}_A X_A)$ consists of:
\begin{enumerate}
\item an inclusion $\alpha\colon A \to B$ with corresponding conditional expectation  $\Lambda \colon B \to A$;
\item a right $B$-module structure on $X$ which is compatible with $\alpha$ and $\Lambda$ 
in the sense that $x \cdot \alpha(a) = x \cdot a$ for all $x\in X$ and $a\in A$
and $\Lambda((x_1 \mid x_2)_B) = (x_1 \mid x_2)_A$ for all $x_1,x_2\in X$;
\item a left action of $A$ on $X_B$ by adjointable operators that agrees with the left action of $A$ on $X_A$. In either case, we denote the left action by $\phi_X$.
\end{enumerate}
Let $B^\Lambda_A$ be the completion of $B$ with respect to the inner product $\pairing{b_1}{b_2}_A=\Lambda(b^*_1b_2)$ for all $b_1,b_2\in B$,
and let $({\rm Id}_B,{}_B B^\Lambda_A)$ be the associated $B$--$A$-correspondence with left action of $B$ given by multiplication.
We then define the \emph{out-split correspondence} $(\phi_\Lambda, B^{\Lambda} \ox_A X_B)$ over $B$ where the left action is just left multiplication.
\end{defn}

The idea behind \cref{def:nc-outsplit} is that by using the expectation $\Lambda$ we are able to factor the structure of $X_A$ as a right module through the algebra $B$. The following lemma makes this more precise. We write $[b]$ for the class of of $b \in B$ in $B^{\Lambda}$. 

\begin{lemma} \label{lem:outsplit_factorisation}
The correspondence $(\phi_X, {}_A X_A)$ is isomorphic to $(\phi_X\ox \Id_{B^{\Lambda}}, {}_A X_B \ox_B B^\Lambda_A)$.
\end{lemma}
\begin{proof}
Let $x,x' \in X_B$ and $b,b' \in B$. Observe that
\[
(x \cdot b \mid x ' \cdot b')_A = \Lambda(	(x \cdot b \mid x ' \cdot b')_B) = \Lambda(b^* (x \mid x')_B b') = ([b] \mid [(x \mid x')_B b'])_A = (x \ox [b] \mid x' \ox [b'])_A.
\]
In particular $\|x \cdot b\| = 0$ if and only if $\|x \ox [b]\| = 0$. 
Consequently, the map $\beta \colon X_B \ox_B B^{\Lambda} \to X_A$ given by $\beta(x \ox [b]) = x \cdot b$ for $x \in X_A$ and $b \in [b]$ is well-defined. 
The map $\beta$ is clearly an $A$--$A$-bimodule map, and so $(\Id_A, \beta)$ defines an injective correspondence morphism from $(\phi_X\ox \Id_{B^{\Lambda}}, {}_A X_B \ox_B B^\Lambda_A)$ to $(\phi_X, {}_A X_A)$.
For surjectivity fix $x \in X_A$. Then there exists $y \in X_A$ such that $x = y \cdot (y \mid y)_A = \beta(y \ox [\alpha((y \mid y)_A)])$.
\end{proof}

\begin{thm}
\label{thm:cunt}
The correspondence $(\phi_X, {}_A X_A)$ is elementary strong shift equivalent to the out-split $(\phi_\Lambda, B^\Lambda \ox_A X_B)$.
When $(\phi_X, {}_A X_A)$ is regular and nondegenerate, then the Cuntz--Pimsner algebras
$\cO_{X\ox B^\Lambda}$ and $\cO_{B^\Lambda\ox X}$ are gauge equivariantly Morita equivalent.
\end{thm}

\begin{proof}
Appealing to \cref{lem:outsplit_factorisation}, it follows by definition that $(\phi_X, {}_A X_A)$ is elementary strong shift equivalent to  $(\phi_\Lambda, B^\Lambda \ox_A X_B)$. The Morita equivalence is the main result of \cite{MPT} applied to the correspondences $ R = (\phi_X,{}_AX_B)$ and $ S= ({\rm Id}_B,{}_BB^\Lambda_A)$, and the gauge equivariance follows from \cref{thm:sse-equivariant-morita-equivalence}.
\end{proof}

\begin{rmk}
With apologies for the terminology, un-out-splitting seems more natural. That is starting with a correspondence $(A,X_B)$ and an expectation $\Lambda \colon B \to A$, one can naturally construct $(A, X \ox_B B^\Lambda_A)$. 
In our previous language we would have $X_A \cong X_B \ox_B B_A^{\Lambda}$.	
The downside is that $(A,X_B)$ is not a self-correspondence. 
\end{rmk}

In the case where $X_A = X(E)$ is the correspondence of a directed graph $E$ with out-split $\bO$, \cref{def:nc-outsplit} recovers the correspondence of the associated out-split graph $X(E_\bO)$. 

\begin{prop}
Let $\bO = (\alpha, E_{\bO}^0, \psi)$ be an out-split of a regular topological graph $E$. Let $A = C_0(E^0)$ and $B=C_0(E^0_{\bO})$. Then:
\begin{enumerate}
\item $\alpha^* \colon A \to B$ given by $\alpha^*(a)(v) = a(\alpha(v))$ is an injective $*$-homomorphism;
\item the conditional expectation $\Lambda \colon B \to A$ given by  
\[
\Lambda(b)(v) = \sum_{u \in \alpha^{-1}(v)} b(u)
\]
for $b \in C_c(E^0_{\bO})$ is compatible with $\alpha^*$; and
\item\label{3} $X(E)$ can be equipped with the structure of a right $B$-module via the formulae 
\[
(x \cdot b)(e) = x(e) b(\psi(e)) \quad \text{ and } \quad (x \mid y)_B (u) = \sum_{e \in \psi^{-1}(u)} \ol{x(e)} y(e)
\]
for all $x,y \in C_c(E^1)$ and $b \in C_0(E^0_{\bO})$, and the left action of $A$ on $X(E)$ also defines a left action by adjointable operators with respect to the new right $B$-module structure.
\end{enumerate}
Moreover, the correspondences $(\phi,X(E_{\bO}))$ and $(\phi^{\Lambda},  B^{\Lambda} \ox_A X(E) )$ are isomorphic.
\end{prop}

\begin{proof}
Since $\alpha \colon E^0_{\bO} \to E^0$ is  proper and surjective, $\alpha^*$ defines an injective $*$-homomorphism. The expectation $\Lambda$ is clearly compatible with $\alpha^*$ in the sense that $\Lambda(\alpha^*(a_1)b \alpha^*(a_2)) = a_1 \Lambda(b) a_2$ for all $a_1,a_2 \in A$ and $b \in B$. It is also straightforward to verify that the formulae in \ref{3} define a right $B$-module structure on $X(E)$. 

Since $s =  \alpha \circ \psi$, it follows that $x \cdot \alpha^*(a) = x \cdot a$.
Moreover,
\begin{align*}
\Lambda((x_1 \mid x_2)_B)(v) 
&= \sum_{u \in \alpha^{-1}(v)} (x_1 \mid x_2)_B(u) 
= \sum_{u \in \alpha^{-1}(v)}  \sum_{e \in \psi^{-1}(u)} \ol{x_1(e)} x_2(e)\\
&= \sum_{s(e) = v}  \ol{x_1(e)} x_2(e)
= (x_1 \mid x_2)_A(v),
\end{align*}
for all $x_1,x_2\in X$ and $v\in E^0$. It follows that we have an out-split (cf.~\cref{def:nc-outsplit}) on the graph module $X(E)$ so we may form the out-split correspondence $(\phi^{\Lambda}, B^{\Lambda} \ox_A X(E) )$. 

We would like to define a map $\Psi\colon (\phi_\Lambda, B^\Lambda\ox_A X(E)_B) \to (\phi, X(E_{\bO}))$ by 
\[
\Psi([b]\ox x)(u,e) = b(u) x(e),
\]
for all $[b]\ox x\in B^\Lambda\ox_A X_B$ and $(u,e)\in E_{\bO}^1$.
For $u\in E_{\bO}^0$, recall that  
\[
s_{\bO}^{-1}(u) = \{ (w,e)\in E_{\bO}^0\times E^1 : \psi(e) = u, \alpha(w) = r(e) \}.
\]
With this observation we can compute
\begin{align*}
(\Psi ([b_1] \ox x_1) \mid \Psi ([b_2] \ox x_2))_B(u) 
&= \sum_{(w,e)\in s_{\bO}^{-1}(u)} \ol{ b_1(w)x_1(e)} b_2(w) x_2(e)\\
&= \sum_{e\in \psi^{-1}(u)} \sum_{w\in \alpha^{-1}(r(e))} \ol{ b_1(w)x_1(e)} b_2(w) x_2(e) \\
&= \sum_{e \in \psi^{-1}(u)} \ol{x_1(e)} \Lambda(b_1^*b_2) (r(e))x_2(e)\\
&= (x_1 \mid \Lambda(b_1^*b_2)x_2)_B(u)\\
&= ([b_1] \ox x_1 \mid [b_2] \ox x_2)_B (u).
\end{align*}
Consequently, $\Psi$ is well-defined and extends to an isometric linear map  $\Psi \colon B^\Lambda\ox_A X_B \to X(E_{\bO})$.
The map $\Psi$ preserves the left action since
\[
\Psi(\phi_{\Lambda}(b_1) ([b_2] \ox x))(v,e) = \Psi ([b_1b_2] \ox x)(v,e) = b_1(v)b_2(v)x(e) = \phi (b_1) \Psi([b_2] \cdot x) (v,e),
\]
for all $b_1,b_2 \in B$, $x \in X$, and $(v,e)\in E_{\bO}^1$; 
similarly, $\Psi$ preserves the right action as
\[
\Psi([b_1] \ox x \cdot b_2) (v,e) = b_1(v)x(e)b_2(\psi(e)) = (\Psi([b_1] \ox x) \cdot b_2)(v,e),
\]
for all $b_1,b_2\in B$, $x\in X$, and $(v,e)\in E_{\bO}^1$.

Since functions of the form $(v,e) \mapsto b(v) x(e)$ densely span $C_c(E^0_{\bO} \times_{\alpha,r} E^1)$, it follows from the Stone-Weierstrass theorem that $\Psi$ is surjective.
\end{proof}

\begin{example}
\label{eg:n-m-circles-again}
We give the out-split version of \cref{eg:n-m-circles}.

Fix $m,n \in \Z \setminus \{0\}$ and let $E^0 \coloneqq \bT$ and $E^1 \coloneqq \bT$. 
Define $r,s \colon E^1 \to E^0$ by $r(z) = z^m$ and $s(z) = z^n$. Then $E = (E^0,E^1,r,s)$ is a topological graph.
Suppose $a,b \in \Z$ satisfy $n = ab$. Define $\psi \colon E^1 \to \bT$ by $\psi(z) = z^a$ and $\alpha \colon \bT \to E^0$ by $\alpha(z) = z^b$. 
Since $s(z) = z^n = (z^{a})^b = \alpha \circ \psi(z)$, it follows that $\bO = (\alpha, \bT, \psi)$ is an out-split of $E$. 
Exactly as in \cref{eg:n-m-circles}, the new edge space
\[
E^1_\bO = \{ (z_1,z_2) \in \bT^2 \mid  z_1^b=z_2^m \}.
\] 
is homeomorphic to a disjoint union of $\gcd(m,b)$ copies of $\bT$. 

An explicit identification of $E^1_\bO$ with the disjoint union of circles is given by fixing a primitive $|b|$-th root of unity $\lambda$. Let $\pi \colon \{1 ,\ldots, \gcd(m,b)\} \times \bT \to E_\bO^1$ be the homeomorphism defined by $\pi(k,z)  = (\lambda^kz^{m/\gcd(m,b)}, z^{b/\gcd(m,b)})$. Under this identification, 
\[
r_\bO(k,z) = \lambda^kz^{m/\gcd(m,b)}  \quad \text{and} \quad s_\bO(k,z) = \psi(z^{b/\gcd(m,b)})=z^{ab/\gcd(m,b)}=z^{n/\gcd(m,b)}.
\]
By \cref{thm:cunt}, the topological graphs $E$ and $E_\bO$ have gauge equivariantly Morita equivalent $C^*$-algebras. This is very different from the $*$-isomorphism arising from the analogous in-split of the range map.
\end{example}

\end{document}